\documentclass[a4paper, reqno,11pt]{amsart}

\usepackage{amsthm,amsfonts,amssymb,amsmath,amsxtra}
\usepackage[all]{xy}
\SelectTips{cm}{}
\usepackage{tikz}
\usepackage{verbatim}

\usepackage[margin=1.4in]{geometry}
\usepackage{mathrsfs}

\usepackage{mathdots}
\usepackage{MnSymbol}

\RequirePackage{xspace}
\RequirePackage{etoolbox}
\RequirePackage{varwidth}
\RequirePackage{enumitem}
\RequirePackage{tensor}
\RequirePackage{mathtools}
\RequirePackage{longtable}
\RequirePackage{multirow}
\RequirePackage{makecell}
\RequirePackage{bigints}
\RequirePackage{xcolor}

\usepackage[backref=page]{hyperref} %

\newcommand{\sE}{\ensuremath{\mathscr{E}}\xspace}

\newcommand{\fkq}{\ensuremath{\mathfrak{q}}\xspace}

\newcommand{\BC}{\ensuremath{\mathbb{C}}\xspace}
\newcommand{\BD}{\ensuremath{\mathbb{D}}\xspace}
\newcommand{\BE}{\ensuremath{\mathbb{E}}\xspace}
\newcommand{\BF}{\ensuremath{\mathbb{F}}\xspace}
\newcommand{\BG}{\ensuremath{\mathbb{G}}\xspace}

\newcommand{\BL}{\ensuremath{\mathbb{L}}\xspace}

\newcommand{\BQ}{\ensuremath{\mathbb{Q}}\xspace}
\newcommand{\BR}{\ensuremath{\mathbb{R}}\xspace}

\newcommand{\BV}{\ensuremath{\mathbb{V}}\xspace}
\newcommand{\BW}{\ensuremath{\mathbb{W}}\xspace}
\newcommand{\BX}{\ensuremath{\mathbb{X}}\xspace}

\newcommand{\BZ}{\ensuremath{\mathbb{Z}}\xspace}

\newcommand{\bH}{\ensuremath{\mathbf{H}}\xspace}

\newcommand{\bM}{\ensuremath{\mathbf{M}}\xspace}
\newcommand{\bN}{\ensuremath{\mathbf{N}}\xspace}

\newcommand{\CD}{\ensuremath{\mathcal{D}}\xspace}
\newcommand{\CE}{\ensuremath{\mathcal{E}}\xspace}
\newcommand{\CF}{\ensuremath{\mathcal{F}}\xspace}

\newcommand{\CL}{\ensuremath{\mathcal{L}}\xspace}
\newcommand{\CM}{\ensuremath{\mathcal{M}}\xspace}
\newcommand{\CN}{\ensuremath{\mathcal{N}}\xspace}
\newcommand{\CO}{\ensuremath{\mathcal{O}}\xspace}

\newcommand{\CV}{\ensuremath{\mathcal{V}}\xspace}

\newcommand{\CX}{\ensuremath{\mathcal{X}}\xspace}

\newcommand{\CZ}{\ensuremath{\mathcal{Z}}\xspace}

\newcommand{\heart}{{\heartsuit}}

\newcommand{\spade}{{\spadesuit}}

\DeclareMathOperator{\Aut}{Aut}

\DeclareMathOperator{\charac}{char}

\newcommand{\del}{\operatorname{\partial Orb}}

\DeclareMathOperator{\diag}{diag}

\DeclareMathOperator{\End}{End}

\newcommand{\Fil}{\ensuremath{\mathrm{Fil}}\xspace}

\DeclareMathOperator{\Gal}{Gal}

\newcommand{\GL}{\mathrm{GL}}

\newcommand{\gr}{\mathrm{gr}}

\newcommand{\GU}{\mathrm{GU}}

\DeclareMathOperator{\Hom}{Hom}

\newcommand{\id}{\ensuremath{\mathrm{id}}\xspace}

\DeclareMathOperator{\Int}{\ensuremath{\mathrm{Int}}\xspace}

\DeclareMathOperator{\Lie}{Lie}

\newcommand{\loc}{\ensuremath{\mathrm{loc}}\xspace}

\newcommand{\Nilp}{\ensuremath{\mathtt{Nilp}}\xspace}

\DeclareMathOperator{\Orb}{Orb}

\DeclareMathOperator{\rank}{rank}
\DeclareMathOperator{\Ros}{Ros}

\newcommand{\Sch}{\ensuremath{\mathtt{Sch}}\xspace}
\newcommand{\FL}{\ensuremath{\CF\!\CL}\xspace}

\newcommand{\red}{\ensuremath{\mathrm{red}}\xspace}

\DeclareMathOperator{\Res}{Res}
\newcommand{\rs}{\ensuremath{\mathrm{reg}}\xspace}

\DeclareMathOperator{\Spec}{Spec}
\DeclareMathOperator{\Spf}{Spf}

\newcommand{\SO}{{\mathrm{SO}}}

\newcommand{\SU}{{\mathrm{SU}}}

\DeclareMathOperator{\tr}{tr}

\newcommand{\U}{\mathrm{U}}

\newcommand{\Ver}{\mathrm{Vert}}

\DeclareMathOperator{\vol}{vol}

\newcommand{\wt}{\widetilde}

\newcommand{\pair}[1]{\langle {#1} \rangle}

\newcommand{\ov}{\overline}
\newcommand{\bre}{\breve}
\newcommand{\incl}{\hookrightarrow}
\newcommand{\lra}{\longrightarrow}

\newcommand{\lv}{\lvert}
\newcommand{\rv}{\rvert}

\newtheorem{theorem}{Theorem}
\newtheorem{proposition}[theorem]{Proposition}
\newtheorem{lemma}[theorem]{Lemma}

\newtheorem{conjecture}[theorem]{Conjecture}
\newtheorem{`conjecture'}[theorem]{``Conjecture''}

\theoremstyle{definition}
\newtheorem{definition}[theorem]{Definition}
\newtheorem{example}[theorem]{Example}

\newtheorem{remark}[theorem]{Remark}

\newenvironment{altenumerate}
   {\begin{list}
      {(\theenumi) }
      {\usecounter{enumi}
       \setlength{\labelwidth}{0pt}
       \setlength{\labelsep}{0pt}
       \setlength{\leftmargin}{0pt}
       \setlength{\itemsep}{\the\smallskipamount}
       \renewcommand{\theenumi}{\roman{enumi}}
      }}
   {\end{list}}

\numberwithin{equation}{section}
\numberwithin{theorem}{section}

\setitemize[0]{leftmargin=*,itemsep=\the\smallskipamount}
\setenumerate[0]{leftmargin=*,itemsep=\the\smallskipamount}

\renewcommand{\to}{%
   \ifbool{@display}{\longrightarrow}{\rightarrow}%
   }
\let\shortmapsto\mapsto
\renewcommand{\mapsto}{%
   \ifbool{@display}{\longmapsto}{\shortmapsto}%
   }
\newcommand{\hooklongrightarrow}{\mathrel{\mkern 0.5mu\lhook\mkern -3.5mu\relbar\mkern -3mu \rightarrow }}
\newcommand{\inj}{%
   \ifbool{@display}{\hooklongrightarrow}{\hookrightarrow}
   }
\newcommand{\isoarrow}{%
   \ifbool{@display}{\overset{\sim}{\longrightarrow}}{\xrightarrow\sim}%
   }
\newlength{\olen}
\newlength{\ulen}
\newlength{\xlen}
\newcommand{\xra}[2][]{%
   \ifbool{@display}%
      {\settowidth{\olen}{$\overset{#2}{\longrightarrow}$}%
       \settowidth{\ulen}{$\underset{#1}{\longrightarrow}$}%
       \settowidth{\xlen}{$\xrightarrow[#1]{#2}$}%
       \ifdimgreater{\olen}{\xlen}%
          {\underset{#1}{\overset{#2}{\longrightarrow}}}%
          {\ifdimgreater{\ulen}{\xlen}%
             {\underset{#1}{\overset{#2}{\longrightarrow}}}
             {\xrightarrow[#1]{#2}}}}%
      {\xrightarrow[#1]{#2}}
   }
\makeatother
\newcommand{\xyra}[2][]{%
   \settowidth{\xlen}{$\xrightarrow[#1]{#2}$}%
   \ifbool{@display}%
      {\settowidth{\olen}{$\overset{#2}{\longrightarrow}$}%
       \settowidth{\ulen}{$\underset{#1}{\longrightarrow}$}%
       \ifdimgreater{\olen}{\xlen}%
          {\mathrel{\xymatrix@M=.12ex@C=3.2ex{\ar[r]^-{#2}_-{#1} &}}}%
          {\ifdimgreater{\ulen}{\xlen}%
             {\mathrel{\xymatrix@M=.12ex@C=3.2ex{\ar[r]^-{#2}_-{#1} &}}}
             {\mathrel{\xymatrix@M=.12ex@C=\the\xlen{\ar[r]^-{#2}_-{#1} &}}}}}%
      {\mathrel{\xymatrix@M=.12ex@C=\the\xlen{\ar[r]^-{#2}_-{#1} &}}}%
   }
\makeatletter
\newcommand{\xla}[2][]{%
   \ifbool{@display}%
      {\settowidth{\olen}{$\overset{#2}{\longleftarrow}$}%
       \settowidth{\ulen}{$\underset{#1}{\longleftarrow}$}%
       \settowidth{\xlen}{$\xleftarrow[#1]{#2}$}%
       \ifdimgreater{\olen}{\xlen}%
          {\underset{#1}{\overset{#2}{\longleftarrow}}}%
          {\ifdimgreater{\ulen}{\xlen}%
             {\underset{#1}{\overset{#2}{\longleftarrow}}}
             {\xleftarrow[#1]{#2}}}}%
      {\xleftarrow[#1]{#2}}
   }
\renewcommand{\lra}{%
   \ifbool{@display}{\longleftrightarrow}{\leftrightarrow}%
   }

\begin{document}
\thanks{Research of W. Zhang is partially supported by the NSF grant DMS  \#1901642. }

\title[AFL for Bessel subgroups]{More Arithmetic Fundamental Lemma conjectures: the case of Bessel subgroups}

\author{W. Zhang}
\address{Massachusetts Institute of Technology, Department of Mathematics, 77 Massachusetts Avenue, Cambridge, MA 02139, USA}
\email{weizhang@mit.edu}

\date{\today}

\begin{abstract}
We  define some formal moduli space of quasi-isogenies of isoclinic $p$-divisible groups with a non-reductive group as the ``structure group". We then formulate new Arithmetic Fundamental Lemma conjectures for Bessel subgroups in the context of the arithmetic Gan--Gross--Prasad conjectures. Some (very limited) evidence is presented. \end{abstract}

\maketitle

\tableofcontents

\section{Introduction}\label{s:intro}

For a (smooth projective) algebraic variety over a number field, the vanishing of its Hasse--Weil L-function at  the central point is conjectured to be accounted for by the non-triviality of its Chow group of (homologically trivial) algebraic cycles. A notable example is the conjecture of Birch and Swinnerton-Dyer for elliptic curves. There is now more evidence in higher dimensional cases provided by special cycles on Shimura varieties. In this context, the first example is provided by the Gross--Zagier formula \cite{GZ,YZZ1}, which relates the N\'eron--Tate heights of Heegner divisors on modular curves to the first central derivative of the relevant $L$-functions. The arithmetic Gan--Gross--Prasad (GGP) conjecture \cite{GGP,RSZ3,SZ,Z12} is one of the generalizations of the Gross--Zagier formula to higher-dimensional Shimura varieties.  The relevant Shimura variety is associated to the product unitary group $G=\U(1,n-2)\times \U(1,n-1)$ or orthogonal group $G=\SO(2,n-2)\times \SO(2,n-1)$, on which there is the so-called arithmetic diagonal cycle associated to the diagonally embedded $H=\U(1,n-2)$ or $H=\SO(2,n-2)$. Then the case $n=2$ (for both unitary and orthogonal groups) essentially recovers Heegner divisors on modular curves.  

The  arithmetic Gan--Gross--Prasad conjecture is inspired by the (usual) Gan--Gross--Prasad conjecture, which relates automorphic period integrals on classical groups to the special values of Rankin--Selberg tensor product $L$-functions (cf. for recent advances in this direction). The latter conjecture makes sense for more general groups, such as the product group $G=\U(m)\times \U(n)$ and  $G=\SO(m)\times \SO(n)$ when $m-n$ is odd,  while the arithmetic conjecture only exists for the case $|n-m|=1$\footnote{Yifeng Liu \cite{Liu18} discovered how to formulate the arithmetic conjecture in the equal rank unitary case $\U(n)\times \U(n)$; however we will not discuss the even $n-m$ case in this paper.}. A natural question is how to formulate the arithmetic conjecture for the general product groups. This paper aims to give a partial answer to this question. In short, there is no global cycle to formulate a global conjecture but there is a local cycle and one can formulate a local conjecture. In other words, there is no Shimura variety with the desired special algebraic cycle when $|n-m|>1$, but there is a local Shimura variety (\cite{RV,RZ}, a formal scheme in this paper) over a $p$-adic integer ring with the desired cycle (a  closed formal subscheme) to formulate an arithmetic fundamental lemma (AFL) conjecture.

To be more precise let us recall that 
Jacquet and Rallis  formulated a relative trace formula (RTF) approach to the (usual) Gan--Gross--Prasad conjecture for $\U(n-1)\times\U(n)$ \cite{JR}.
Inspired by their approach, in \cite{Z12} the author proposed a relative trace formula approach to the arithmetic Gan--Gross--Prasad conjecture.  A key local statement of this approach is  the arithmetic fundamental lemma (AFL) conjecture,  which relates the central derivatives of certain orbital integrals to the arithmetic intersection numbers of ``local arithmetic diagonal cycles" on a Rapoport--Zink (RZ) space for unitary groups,
\begin{align}\label{eq:AFL}
\del\bigl(\gamma, \mathbf{1}_{S_n(O_{F_0})}\bigr) 
	   = -\Int(g)\cdot\log q,
\end{align}
where $F_0$ is a $p$-adic local field and $S_n\simeq \GL_{n,F}/\GL_{n,F_0}$ is the symmetric space with respect to an unramified quadratic extension $F$ of $F_0$. See \S\ref{ss:AFL} for the precise formulation.
This conjecture is largely proved by the author \cite{Z21} when $F_0=\BQ_p$, and by Mihatsch and the author in \cite{MZ} for general $p$-adic fields $F_0$, provided the residue cardinality $q\geq n$ of $F_0$ in both cases. 

In \cite{Liu14}, Liu generalized the RTF construction of Jacquet and Rallis to all unitary groups $\U(m)\times \U(n)$ (the Bessel case when $n-m$ is odd, and Fourier--Jacobi case when $n-m$ is even). In particular, he has essentially generalized the left hand side of the AFL conjecture \eqref{eq:AFL} above to the general case of unitary groups; we will recall  the Bessel case of his orbital integrals in \S\ref{s:FL}. We will also give a lattice counting interpretation of the orbital integral in \S\ref{ss:lat} and prove the conjecture for some special elements (by reduction to Jacquet--Rallis FL). Then we will generalize the local arithmetic diagonal cycle to the general Bessel case, see \S\ref{s:AFL}. The key new input is a class of RZ spaces with non-reductive groups as the ``structure groups", which may be of independent interest. The RZ spaces for reductive groups (in the EL or PEL cases) parameterize $p$-divisible groups with additional structures, cf. \S\ref{ss:RZ data}. Here we add a filtration by $p$-divisible groups (respecting the additional structure) to define the RZ spaces for certain non-reductive groups, see \S\ref{s:RZ fil}.

There are vast generalizations of the GGP conjecture in terms of period integrals on  (homogeneous) {\em spherical varieties} $G/H$ for a pair of algebraic groups  $(G,H), H\subset G,$ over number fields; see the series of work by Sakellaridis (e.g. \cite{Sa}) and  a version of the general conjecture of Ichino--Ikeda type, formulated by Sakellaridis and Venkatesh \cite{SV}. The idea  in this paper seems applicable to more spherical pairs $(G,H)$ where $G/H$ is non-affine (equivalently, by a theorem of Matsushima, $H$ is non-reductive; here $G$ is always assumed to be reductive). In fact, 
there are many more non-affine (homogeneous) spherical varieties than affine ones, and it would be interesting to classify all such spherical varieties which behaves like the cases in this paper. For all of the non-affine ones, it is currently hopeless to construct the global cycles but we expect to find the desired 
local cycles over $p$-adic fields, at least assuming the existence of a ``good" theory of (integral models of) local Shimura varieties (i.e., one needs a suitable intersection theory on them). 
Along this direction, we hope to investigate the arithmetic GGP conjecture for $G=\SO(m)\times \SO(n)$ when $m-n$ is odd, which involves local Shimura varieties of Hodge type (GSpin RZ spaces). In the non-affine case, the arguably  ``simplest" example is the Ginzburg--Rallis subgroup, which only involves RZ spaces of EL type and is a natural generalization of the Lubin--Tate deformation space; we will discuss this briefly in \S\ref{s:GR}.  There are also non-affine spherical varieties for  exceptional groups, which are more challenging at this moment, due to the lack of a ``good" theory of integral model of local Shimura varieties.

\subsection*{Acknowledgements} With admiration we dedicate this paper to Dick Gross on the occasion of his 70th birthday. He is a pioneer on the theory of special cycles and their relation to  L-functions, the main subject of this paper. The author is grateful for his generosity sharing many ideas and insights, and for his constant encouragement over the years. 
The author thanks Yifeng Liu, Michael Rapoport, Yiannis Sakellaridis, Jonathan Wang, and Shou-Wu Zhang for their comments on an earlier draft.
\section{RZ spaces: the reductive case}
Rapoport and Zink have constructed formal moduli spaces of $p$-divisible groups with EL or PEL structures (see \cite[\S3]{RZ}).  These moduli spaces admit actions by (the $\BQ_p$-points of) certain reductive groups over $\BQ_p$.  We recall their construction in this section.

\subsection{The RZ datum: the EL and PEL case}\label{ss:RZ data}
 The formal moduli spaces of Rapoport and Zink, abbreviated as RZ spaces in this paper, depend on some linear algebraic data, called ``RZ data". In this paper, we will only consider {\em simple} RZ data, cf. \cite[\S4.1]{RV}.  A {\em simple rational RZ datum} in the EL case is a tuple 
$$
\CD=(F,B,V,\{\mu\},[b]),
$$  where
\begin{itemize}
\item $F$ is a finite field extension of $\BQ_p$,
\item $B$ is a central {\em division} $F$-algebra,
\item $V$ is a finite dimensional left $B$-module,
\item $\{\mu\}$ is a conjugacy class of minuscule cocharacters $\mu:\BG_m\to G_{\ov\BQ_p}$ where we define $G=\GL_B(V)$ (as an algebraic group over $\BQ_p$),
\item $[b]\in B(G)$, where $B(G)$ is the set of $\sigma$-conjugacy classes of $G(\bre \BQ_p)$ (cf. \cite[\S2.1]{RV}). 
\end{itemize}
Here, we denote the completion of a maximal unramified extension of a $p$-adic field $F$ by $\bre F$. Moreover, $\sigma\in\Aut(\bre \BQ_p/\BQ_p)$ denotes the Frobenius automorphism.

A {\em simple rational RZ datum} in the PEL case is a tuple $$\CD=(F,B,V,(\cdot,\cdot),\ast,\{\mu\},[b])$$
 where $F,B,V,\{\mu\},[b]$ are as in the EL case  (but for the structure group $G$ defined below) and 
\begin{itemize}
\item $(\cdot,\cdot): V\times V\to \BQ_p$ is a non-degenerate alternating $\BQ_p$-bilinear form,
\item $\ast: B\to B$ is an involution such that $(ax,y)=(x,a^\ast y)$  holds for $a\in B$ and $x,y\in V$,
\item  $G=\GL_B(V, (\cdot,\cdot))$ (as an algebraic group over $\BQ_p$) whose $R$-points for any $\BQ_p$-algebra $R$ are 
$$G(R)=\left\{ g\in \GL_B(V\otimes_{\BQ_p} R)\mid (gx,gy)=c(g)(x,y),c(g)\in R^\times \right\}.$$
\end{itemize}

In both cases, we imposes the following condition on $\{\mu\}$. Any cocharacter  $\mu$ induces a weight decomposition on $V\otimes_{\BQ_p}\ov \BQ_p$. We require that for any $\mu \in\{\mu\}$,  only weights 0 and 1 occur in the decomposition, i.e. 
\begin{align}\label{eq:wt dec}
V\otimes_{\BQ_p}\ov \BQ_p=V_0\oplus V_1.
\end{align} In the PEL case, we further require that for any  $\mu \in\{\mu\}$, the composition with the similitude character $c$
$$
\xymatrix{\BG_m\ar[r]^\mu& G_{\ov \BQ_p}\ar[r]^c& \BG_{m,\ov\BQ_p}}
$$
is the identity. 

Now we turn to the integral datum. A {\em simple integral RZ datum} $\CD_{\BZ_p}$ in the EL case consists, in addition to the rational data $\CD$, of a maximal order $O_B$ in $B$ and an $O_B$-stable lattice $\Lambda$ in $V$. A {\em simple integral RZ datum} $\CD_{\BZ_p}$ in the PEL case consists, in addition to the rational data $\CD$, of a maximal order $O_B$ in $B$ which is stable under the involution $\ast$ and an $O_B$-stable lattice $\Lambda$ in $V$ such that $\Lambda\subset\Lambda^\vee\subset \varpi^{-1}\Lambda$.
Here $\varpi$ denotes a uniformizer of $O_B$, and $$\Lambda^\vee=\{x\in V\mid (x,\Lambda)\subset \BZ_p\}$$ denotes the ``dual lattice" of $\Lambda$\footnote{We may call a lattice $\Lambda$ a {\em vertex lattice} if it satisfies $\Lambda\subset\Lambda^\vee\subset \varpi^{-1}\Lambda$.}.

\begin{remark}More general RZ data (still in the EL and PEL case) can be found in \cite[\S3]{RZ}, where $B$ could be a semisimple $\BQ_p$-algebra and $\Lambda$ is replaced by a ``periodic lattice chain".
\end{remark}

For simplicity we will assume that $[b]$ is {\em basic} for the rest of this paper, cf. \cite[\S2.1]{RV}. Notice that there is a unique basic element in the subset $B(G,\mu)$ of $\mu$-neutral acceptable elements in $B(G)$ (see \cite[Def.~2.3]{RV}). In the EL and PEL case, the basic element can be characterized by the associated isocrystal (given by $b(1_V\otimes \sigma)$) on $V\otimes_{\BQ_p}\bre\BQ_p$ being {\em isoclinic} (i.e., only a single slope appears).  

We let $E=E_{\{\mu\}}$ denote the reflex field (i.e., the field of definition of $\{\mu\}$).
To each $b\in G(\bre \BQ_p)$ one may associate an algebraic group over $\BQ_p$, denoted by $J_b$, such that for any $\BQ_p$-algebra $R$,
\begin{align}\label{eq:Jb}
J_b(R)=\{ g\in G(R\otimes_{\BQ_p} \bre\BQ_p) \mid g (b\sigma)= (b\sigma) g\}.
\end{align} 
Up to isomorphism the group $J_b$ depends only on the class $[b]\in B(G)$. For $[b]$ basic,  the group $J_b$ is an inner form of $G$.

We also single out the (simple) {\em unramified} case \cite[\S3.82]{RZ}, which is of particular interest in this paper. We recall that this means that $B=F$ is an unramified field extension of $\BQ_p$, and in the PEL case, $\Lambda=\Lambda^\vee$ (i.e., a self-dual lattice with respect to the given alternating form on $V$).

We will mostly concentrate on two (families of) examples.
\begin{example}[EL case]\label{ex:EL}
 All of the simple unramified integral RZ data $\CD_{\BZ_p}$ in the EL case arise as follows. Let $B=F$ is an unramified field extension of $\BQ_p$,  $V=F^{n}$ is the standard $n$-dimensional $F$-vector space, $\Lambda\subset V$ an $O_F$-lattice. Then $G=\Res_{F/\BQ_p}\GL_{n,F}$. A cocharacter $$\mu: \BG_m\to G_{\ov\BQ_p}\simeq \prod_{\varphi\in \Hom_{\BQ_p}(F,\ov\BQ_p)}\GL_{n,\ov\BQ_p} $$ consists of a tuple $  (\mu_\varphi)_{\varphi\in \Hom_{\BQ_p}(F,\ov\BQ_p)}$ of cocharacters $\mu_\varphi:   \BG_m\to \GL_{n,\ov\BQ_p}$. We consider (the conjugacy class of)
$$
\mu_\varphi(z)= \begin{bmatrix} {\bf 1}_{r_\varphi}  &\\
&  z \cdot {\bf1}_{s_\varphi}
 \end{bmatrix},
$$
where $(r_\varphi, s_{\varphi})_{\varphi\in \Hom_{\BQ_p}(F,\ov\BQ_p)}$ is  a tuple of non-negative integers such that $r_\varphi+s_{\varphi}=n$.
The reflex field $E$ is then the subfield of $\ov\BQ_p$ corresponding to the open subgroup of $\Gal(\ov\BQ_p/\BQ_p)$ fixing the tuple  $(r_\varphi, s_{\varphi})_{\varphi\in \Hom_{\BQ_p}(F,\ov\BQ_p)}$ (under the obvious action). Let $[b]$ be the unique basic element in $B(G,\mu)$.  

\end{example}
\begin{example}[unitary PEL case]\label{ex:PEL}
We consider the following  simple  integral RZ data $\CD_{\BZ_p}$ in the PEL case, related to unitary groups. Let $F/F_0$ be an unramified quadratic extension of $p$-adic fields with the non-trivial Galois involution denoted by $a\mapsto \ov{a}$. Let $B=F$, and $V$ a non-degenerate $F/F_0$-Hermitian space of $F$-dimension $n$. Upon choosing a unit in the ``imaginary" $F_0$-line $F^-:=\{a\in F\mid  \ov a=-a\}$, the Hermitian form induces an alternating $\BQ_p$-bilinear form on $V$. Then $G=\GU(V,(\cdot,\cdot))$ is the unitary similitude group. Let $\Lambda\subset V$ be an $O_F$-lattice such that $\Lambda\subset\Lambda^\vee\subset \varpi^{-1}\Lambda$. We consider the conjugacy class of $\mu: \BG_m\to G_{\ov\BQ_p}$ corresponding to the ``signature" 
$(r_\varphi, s_{\varphi})_{\varphi\in \Hom_{\BQ_p}(F,\ov\BQ_p)}$ such that $r_\varphi+s_{\varphi}=n$ and $r_\varphi=s_{\ov\varphi}$, where $r_\varphi$ denotes the dimension of $V_{0,\varphi}$.  Let $[b]$ be the unique basic element in $B(G,\mu)$.  
\end{example}

\subsection{The formal moduli spaces in the reductive case}\label{ss:RZ red}
Let $O_{E}$ and $O_{\bre E}$ denote the ring of integers in $E$ and $\bre E$ respectively. Let $\BF$ be the residue field of $O_{\bre E}$. Let $\Nilp=\Nilp_{O_{\bre E}}$ denote the category of $O_{\bre E}$-schemes $S$ on which $p$ is locally nilpotent.  For $S\in \Nilp$ we denote by $\ov S$ its special fiber $S\times_{\Spec O_{\bre E}}\Spec O_{\bre E}/(p) $. Note that $\Spec O_{\bre E}/(p) $ is a scheme over $\Spec \BF$.

Fix a simple integral RZ data $\CD_{\BZ_p}$.
We will consider pairs $(X,\iota)$ in the EL case, and triples $(X,\iota,\lambda)$ in the PEL case, where 
\begin{itemize}
\item $X$ is a $p$-divisible group over $S\in \Nilp$, 
\item $\iota:O_B\to\End(X)$ is an action of $O_B$ on $X$,
\item In the PEL case, a polarization $\lambda: X\to X^\vee$.
\end{itemize}
We impose the Kottwitz condition, i.e., the equality of characteristic polynomials for all $ a\in O_B$
$$
 \charac \bigl(\iota(a);\Lie(X)\bigr)= \charac \bigl(a; V_0\bigr),
$$ 
where $V_0$ is the weight zero subspace of $V$, cf. \eqref{eq:wt dec}. In the PEL case, we further impose the condition that $\ker(\lambda)\subset X[\varpi]$ has order equal to $\# (\Lambda^\vee/\Lambda)$, and that the Rosati involution on $O_B$ coincides with $\ast$.

 To define the moduli space, besides an integral RZ datum $\CD_{\BZ_p}$, we also need to fix a {\em framing object}, i.e., a pair $(\BX,\iota_\BX)$ over $\Spec\BF$  in the EL case (resp. a triple $(\BX,\iota_\BX,\lambda_\BX)$ in the PEL case) as above. We assume that the induced isocrystal with the $O_B$-action (and the alternating form induced by the polarization in the PEL case) is isomorphic to $(V\otimes_{\BQ_p}\bre\BQ_p, b\sigma)$ (in the PEL case, $(V\otimes_{\BQ_p}\bre\BQ_p, (\cdot,\cdot),b\sigma)$, the alternating form preserved up to a $\bre\BQ_p^\times$-factor). Then the group of self quasi-isogenies of $\BX$ respecting the additional structure is isomorphic to the group $J_b(\BQ_p)$. 

We are ready to state the definition of the moduli functor of Rapoport--Zink in the EL case:
$$
\xymatrix{\CM: \Nilp\ar[r]& Sets}.
$$
It associates to $S\in\Nilp$ the set of isomorphism classes $(X,\iota, \rho)$ where $(X,\iota)$ is as above and $$
\xymatrix{\rho: X\times_S \ov S\ar[r]& \BX\times_{\Spec\BF}\ov S }
$$ is an $O_B$-linear quasi-isogeny \cite[\S2]{RZ}. In the PEL case, $\CM(S)$ is the set of isomorphism classes $(X,\iota, \lambda,\rho)$, where $(X,\iota,\lambda)$ is as above and $$
\xymatrix{\rho: X\times_S \ov S\ar[r]& \BX\times_{\Spec\BF}\ov S }
$$
 is an $O_B$-linear quasi-isogeny that preserves the polarizations up to a factor in
$\BQ_p^\times$, locally on $\ov S$. 

Here the datum $\rho$ is called a {\em framing} of  $(X,\iota)$ or $(X,\iota,\lambda)$. The group $J_b(\BQ_p)$ acts on $\CM$ by changing the framing $\rho$.
\begin{theorem}[Rapoport--Zink]\label{th:RZ}
The functor $\CM$ is (pro-)representable by a formal scheme, formally locally of finite type over $\Spf O_{\bre E}$.
\end{theorem}
See \cite[Thm.~3.25]{RZ} for a more general statement (cf. \cite[Def.~3.21]{RZ} and the remarks in \S3.23 of {\em loc. cit.} for the equivalence of the definitions).

\subsection{Local models}\label{ss:LM}
The local structure of $\CM$ is governed by the so-called {\em local model} \cite[\S3.26--3.35]{RZ}. For brevity we recall the definition in the unramified  $\CD_{\BZ_p}$ case.

We consider the following functor from the category $\Sch_{O_E}$ of schemes over $O_E$
$$
\xymatrix{\bM^{\loc}: \Sch_{O_E}\ar[r]& Sets}
$$which associates to $S\in \Sch_{O_E}$ the set of isomorphism classes of the following data:
\begin{itemize}
\item  A locally free $\CO_B\otimes_{\BZ_p}\CO_S$-module $\CF$ on $S$: 

\item  A homomorphism of  $\CO_B\otimes_{\BZ_p}\CO_S$-modules 
$$t:\Lambda \otimes_{\BZ_p} \CO_S\to \CF.
$$
\end{itemize}
We require the following conditions to hold:
\begin{itemize}
\item   the action of $O_B$ satisfies the analog of Kottwitz condition
$$
{\rm char}(\iota(a);\CF)={\rm char}(a; V_0),\quad\forall a\in O_B.
$$ 

\item The homomorphisms $t$ is surjective.
\item In the PEL case, the isomorphism $\Lambda^\ast\simeq \Lambda$ induces the following commutative diagram
$$\xymatrix{ \Lambda\otimes_{\BZ_p} \CO_S\ar[d]^{t}\ar[r]^\sim&(\Lambda \otimes_{\BZ_p} \CO_S)^*\ar[d] \\
\CF\ar[r]^\sim & (\ker(t))^*.}
$$
(Note that the polarization is assumed to be principal.) Here $\ast$ denotes the $\CO_S$-linear dual.
\end{itemize}

It is easy to see that the functor $\bM^{\loc}$ is represented by a projective scheme over $\Spec O_{ E}$, being a closed subscheme of a product of Grassmannians. In our unramified case, it is smooth.

Let $\hat\bM^{\loc}$ denote the $p$-adic completion of the base change $\bM^{\loc}\times_{\Spec O_{ E}}\Spec O_{ \bre E}$; it is a formal scheme over $\Spf O_{ \bre E}$. Then Grothendieck--Messing theory implies that any point of $\CM$ has an \'etale neighborhood which is formally \'etale  over $\hat\bM^{\loc}$  (see \cite[\S3.32]{RZ}).
Therefore, the smoothness of the local model  $\bM^{\loc}\to\Spec O_E$ implies that $\CM$ is formally smooth over $\Spf O_{ \bre E}$.  
In the two examples \ref{ex:EL} and \ref{ex:PEL}, if the RZ data are unramified, then relative dimension of $\CM$ over $\Spf O_{\bre E}$ is given by 
\begin{align}\label{eq:dim EL}
\sum_{\varphi\in \Hom_{\BQ_p}(F,\ov\BQ_p)} r_\varphi s_{\varphi}
\end{align}
in the EL case, and 
\begin{align}\label{eq:dim PEL}
\frac{1}{2}\sum_{\varphi\in \Hom_{\BQ_p}(F,\ov\BQ_p)} r_\varphi s_{\varphi}=\sum_{\varphi_0\in \Hom_{\BQ_p}(F_0,\ov\BQ_p)} r_\varphi s_{\varphi}
\end{align}
 in the unitary PEL case,
where $\varphi\in \Hom_{\BQ_p}(F,\ov\BQ_p)$ is any extension of $\varphi_0\in \Hom_{\BQ_p}(F_0,\ov\BQ_p)$.

\begin{example}[The ``totally definite" or ``banal" case]\label{ex:def case}  In Example \ref{ex:EL} and \ref{ex:PEL}, we call an unramified $\CD_{\BZ_p}$ {\em totally definite} if 
$$
r_\varphi s_\varphi=0,\quad \forall \varphi\in \Hom_{\BQ_p}(F,\ov\BQ_p).
$$
Then the relative dimension of $\CM$ over $\Spf O_{\bre E} $ is zero and $J_b\simeq G$. There is an isomorphism of formal schemes
$$
\CM\simeq  \bigsqcup_{G(\BQ_p)/K} \Spf O_{\bre E}
$$
where $K$ is the stabilizer of $\Lambda$.
\end{example}

\section{RZ spaces: the non-reductive case}\label{s:RZ fil}
In this section we construct some generalized RZ spaces in the {\em basic} case, which may be viewed as RZ spaces with the structure group being non-reductive. The new piece of datum defining the moduli problem  is a filtration by $p$-divisible (sub)groups. These formal moduli spaces seem new and they appear to have interesting structure (e.g., even the connected components seem not understood to the author).  The moduli spaces of ``filtered $p$-divisible groups" (with EL or PEL structures)  in the {\em non-basic} case have appeared in Mantovan's work \cite{Man}. Just as in {\em loc. cit.}, here we can establish preliminary properties (representability, local model) of our moduli spaces using the strategy very similar to \cite[\S3]{RZ}.

\subsection{The filtered RZ datum: the EL and PEL case}
\label{ss:fil RZ data}
 A simple rational {\em filtered} RZ datum in the EL case is a tuple
 $$
\CD=(F,B,V, \Fil^\bullet,\{\mu\},[b]),
$$ 
where $F,B,V$ are as in the simple rational RZ datum \S\ref{ss:RZ data} and  
\begin{itemize}
\item $\Fil^\bullet$ is a filtration of finite length 
 $$
0=  \Fil^0V\subset   \Fil^1V\subset \cdots \subset \Fil^\ell V=V
 $$
by $B$-stable submodules $\Fil^i V$, 
\item $\{\mu\}$ is a conjugacy class of minuscule cocharacters $\mu:\BG_m\to H_{\ov\BQ_p}$ where $H:=\GL_B(V,\Fil^\bullet)$ is the stabilizer of the filtration  $ \Fil^\bullet$, hence a parabolic subgroup of $G=\GL_B(V)$.
\item $[b]$ is a $\sigma$-conjugacy class of $H(\bre \BQ_p)$. 
\end{itemize}
The filtered RZ datum $\CD$ induces (unfiltered) RZ data
$$
\Fil^i\CD=(F,B, \Fil^i V,\{\Fil^i\mu\},[\Fil^ib]),\quad i=1, \cdots,\ell,
$$
where $\Fil^iG:=\GL_B(\Fil^i V)$, $\Fil^i\mu: \BG_m\to \Fil^iG_{\ov\BQ_p}$, and $[\Fil^ib]\in B(\Fil^iG)$
 are induced by $\{\mu\}$ and $[b]$. Similarly we have induced RZ data
 $$
\gr^i\CD=(F,B, \gr^i V,\{\gr^i\mu\},[\gr^ib]),\quad i=1, \cdots,\ell,
$$ where $\gr^iV:=\Fil^{i}V/\Fil^{i-1} V$ is the $i$-th graded piece, $\gr^iG:=\GL_B(\gr^i V)$, $\gr^i\mu :  \BG_m\to \gr^iG_{\ov\BQ_p}$, and $[\gr^ib]\in B(\gr^iG)$ are induced by $\{\mu\}$ and $[b]$.
 
A simple  {\em filtered} rational RZ datum in the PEL case is a tuple 
$$
\CD=(F,B,V,\Fil^\bullet,(\cdot,\cdot),\ast,\{\mu\},[b]),
$$
 where $F,B,V, \Fil^\bullet, \{\mu\},[b]$ are as in the filtered EL case (but for $H$ defined below), $(\cdot,\cdot), \ast$ are as in the (unfiltered) PEL case, such that 
\begin{itemize}
\item $\Fil^\bullet$ is {\em self-dual}, in the sense that, for every $i=0,\cdots,\ell$, the space $\Fil^i V$ is the exact annihilator of $\Fil^{\ell-i}V$ under the alternating form $(\cdot,\cdot)$ on $V$,
\item  $H:=\GL_B(V, \Fil^\bullet,(\cdot,\cdot))$, a parabolic subgroup of $G=\GL_B(V,(\cdot,\cdot))$.
\end{itemize}
We have induced RZ data of PEL type 
$$
\Fil^{i,\ell-i}\CD=(F,B, \Fil^{i,\ell-i}V,(\cdot,\cdot),\ast,\{\Fil^{i,\ell-i}\mu\},[\Fil^{i,\ell-i}b]),\quad i=1, \cdots,[\ell/2],
$$
with the induced alternating pairing $(\cdot,\cdot)$ on $ \Fil^{i,\ell-i}V:=\Fil^{\ell-i} V/\Fil^i V$.
We also have RZ data of EL type 
$$
\gr^i\CD=(F,B, \gr^i V,\{\gr^i\mu\},[\gr^ib]),\quad i=1, \cdots,[\ell/2],
$$ 
and, if $\ell$ is odd, an RZ datum of PEL type
$$
\gr^{[(\ell+1)/2]}\CD=(F,B, \gr^{[(\ell+1)/2]} V,(\cdot,\cdot),\ast, \{\gr^{[(\ell+1)/2]}\mu\},[\gr^{[(\ell+1)/2]}b]).
$$

Now we turn to the integral datum. A  simple  integral filtered RZ datum $\CD_{\BZ_p}$ in the EL case consists, in addition to the rational data $\CD$, of a maximal order $O_B$ in $B$ and an $O_B$-stable lattice $\Lambda$ in $V$. A  simple integral filtered RZ datum $\CD_{\BZ_p}$ in the PEL case consists, in addition to the rational data $\CD$, of a maximal order $O_B$ in $B$ which is stable under the involution $\ast$ and an $O_B$-stable lattice $\Lambda$ in $V$ such that $\varpi\Lambda\subset\Lambda^\vee\subset\Lambda$.
In the EL case, the $O_B$-stable $\Lambda$ induces $O_B$-stable lattice $\Fil^i\Lambda:=\Lambda\cap \Fil^iV\subset \Fil^i V$ and $\gr^i\Lambda:=\Fil^i\Lambda /\Fil^{i-1}\Lambda\subset \gr^i V$. Therefore, we obtain induced simple integral RZ data $\Fil^i\CD_{\BZ_p}$ and $\gr^i\CD_{\BZ_p}$ of EL type. Similarly in the PEL case, we define $\Fil^{i,\ell-i}\Lambda:=\Fil^{\ell-i}\Lambda/\Fil^i\Lambda\subset \Fil^{i,\ell-i} V$. Then, by the following lemma, we obtain simple  simple integral RZ data $\Fil^{i,\ell-i}\CD_{\BZ_p}$ and $\gr^i\CD_{\BZ_p}$. Moreover, if $\CD_{\BZ_p}$ is unramified, so are all of the induced data.
\begin{lemma}\label{lem: lat}
In the PEL case, let $\Lambda\subset V$ be an $O_B$-stable lattice such that $\Lambda\subset\Lambda^\vee\subset \varpi^{-1}\Lambda$.
Then, for all $i=1, \cdots,[\ell/2]$, $\Fil^{i,\ell-i}\Lambda$ is an $O_B$-stable lattice in $\Fil^{i,\ell-i}V$ such that 
$$\Fil^{i,\ell-i}\Lambda\subset(\Fil^{i,\ell-i}\Lambda)^\vee\subset \varpi^{-1}\Fil^{i,\ell-i}\Lambda.
$$
Moreover, if $\Lambda=\Lambda^\vee$, then $\Fil^{i,\ell-i}\Lambda=(\Fil^{i,\ell-i}\Lambda)^\vee$.
\end{lemma}

\begin{proof}
For a $\BZ_p$-lattice $L$ we will denote
$L^*=\Hom_{\BZ_p}(L,\BZ_p)$ the $\BZ_p$-linear dual. If $L$ is an $O_B$-lattice, we endow $L^\ast$ the twist of the $O_B$-action by the involution $\ast$. Then the alternating form $(\cdot,\cdot)$ induces an $O_B$-linear map 
$\alpha:\Lambda\to \Lambda^\ast$ and an $O_B$-linear isomorphism $\Lambda^\ast\simeq \Lambda^\vee$.  Since $\Fil^i V$ is totally isotropic, $\alpha$ induces an $O_B$-linear map   $\beta: \Fil^{i}\Lambda\to (\Lambda/ \Fil^{i}\Lambda)^*$, which then  induces an $O_B$-linear map   $\beta^\ast:(\Lambda/ \Fil^{i}\Lambda)^*\to(( \Fil^{i}\Lambda)^\ast)^\ast\simeq \Fil^{i}\Lambda$. We obtain a commutative diagram:
$$
\xymatrix{ 0 \ar[r]&\Fil^{i}\Lambda\ar[d]^\beta \ar[r]&\Lambda\ar[d]^\alpha\ar[r]& \Lambda/ \Fil^{i}\Lambda\ar[r]\ar[d]^{\beta^*}& 0 \\
0 \ar[r]&(\Lambda/ \Fil^{i}\Lambda)^* \ar[r]&\Lambda^*\ar[r]&(\Fil^{i}\Lambda)^\ast\ar[r]& 0.
}
$$
Now note that $\ker(\beta^*)=\Fil^{\ell-i}\Lambda/\Fil^i\Lambda= \Fil^{i,\ell-i}\Lambda$ and hence ${\rm coker}(\beta)\simeq (\Fil^{i,\ell-i}\Lambda)^*\simeq (\Fil^{i,\ell-i}\Lambda)^\vee$.
By the snake lemma, we obtain a long exact sequence of $O_B$-modules
$$
\xymatrix{  0 \ar[r]& \Fil^{i,\ell-i}\Lambda\ar[r]
&(\Fil^{i,\ell-i}\Lambda)^\vee \ar[r]&\Lambda^\vee/\Lambda\ar[r]&{\rm coker}(\beta^*)\ar[r]& 0,
}
$$
or equivalently, 
$$
\xymatrix{  0 \ar[r]& 
(\Fil^{i,\ell-i}\Lambda)^\vee/ \Fil^{i,\ell-i}\Lambda \ar[r]&\Lambda^\vee/\Lambda\ar[r]&{\rm coker}(\beta^*)\ar[r]& 0.
}
$$
If $ \Lambda^\vee/\Lambda$ is annihilated by $\varpi\in O_B$, so is $(\Fil^{i,\ell-i}\Lambda)^\vee/ \Fil^{i,\ell-i}\Lambda$. 
If  $ \Lambda^\vee/\Lambda$ vanishes, so does $(\Fil^{i,\ell-i}\Lambda)^\vee/ \Fil^{i,\ell-i}\Lambda$.

\end{proof}

In both cases, we will again assume that the induced $\sigma$-conjugacy class of $G(\bre \BQ_p)$ is {\em basic}. Since this means (in our EL and PEL case) that the associated isocrystal is isoclinic, the induced  isocrystals on $\Fil^{i,\ell-i}V\otimes_{\BQ_p}\bre\BQ_p$ and $\gr^i V\otimes_{\BQ_p}\bre\BQ_p$ are automatically isoclinic and hence the induced $\sigma$-conjugacy classes $\Fil^{i,\ell-i}b$ and $\gr^ib$  are all basic.

Similar to \eqref{eq:Jb} we can define a functor $J_b$. Though we have not checked the detail, it is reasonable to expect that $J_b$ is represented by a smooth affine group scheme over $\BQ_p$ and it is a parabolic subgroup of $J_{\Fil ^{0,\ell}b}$. Nevertheless, one can verify this assertion in all the examples below \ref{ex:EL fil} and \ref{ex:PEL fil}.

\begin{example}[EL case]\label{ex:EL fil}
Let $B=F, V,G, \Lambda$  be as in Example \ref{ex:EL}.  Let $\Fil^\bullet V$ be any filtration of $V$ and let $H=\GL_B(V,\Fil^\bullet)$. Then we may arrange a cocharacter of $G$ to factor through $\mu: \BG_m\to H$. Let $(r^i_\varphi,s^i_\varphi)_{\varphi\in \Hom_{\BQ_p}(F,\ov\BQ_p)}$ be the invariants associated to the induced data on $\gr^iV$. In particular, we have
 $$
 r_\varphi=\sum_i r^i_\varphi ,\quad s_\varphi=\sum_i s^i_\varphi ,\quad \forall \varphi\in \Hom_{\BQ_p}(F,\ov\BQ_p).
 $$

\end{example}
\begin{example}[unitary PEL case]\label{ex:PEL fil}
Let $B=F, V,(\cdot,\cdot), G, \Lambda$  be as in Example \ref{ex:PEL}. Let $\Fil^\bullet V$ be a self-dual filtration of $V$ and let $H=\GL_B(V,(\cdot,\cdot), \Fil^\bullet)\subset G=\GL_B(V,(\cdot,\cdot))$. Let  $\mu: \BG_m\to H_{\ov\BQ_p}$ be a cocharacter.  Let $(r^i_\varphi,s^i_\varphi)_{\varphi\in \Hom_{\BQ_p}(F,\ov\BQ_p)}$ be the invariants associated to the induced data on $\gr^iV$.   Then  for all $\varphi\in \Hom_{\BQ_p}(F,\ov\BQ_p)$
 $$
 r_\varphi=\sum_i r^i_\varphi ,\quad s_\varphi=\sum_i s^i_\varphi,
 $$ 
 and 
 $$
 r^i_\varphi=s^{\ell+1-i}_{\ov{\varphi}}.
 $$
We remark that there are some restrictions to the existence of above $\Fil$ and $\mu$ due to the restriction to the weights and the compatibility with the similitude character. For example, when $F_0=\BQ_p$ and $ \dim_F V=2$ there does not exist any nontrivial filtered data (besides the ``unfiltered" one) satisfying all of the conditions. 
\end{example}

\subsection{The formal moduli spaces in the non-reductive case}
\label{ss:RZ fil}
We fix a simple integral filtered RZ data $\CD_{\BZ_p}$ for the rest of this section. In view of the  possible complexity of the induced data, as illustrated  by Lemma \ref{lem: lat}, we assume that $\CD_{\BZ_p}$ is {\em unramified}. In particular, $B=F$ is an unramified field extension of $\BQ_p$, the reflex field $E$ is unramified over $\BQ_p$ and hence $\bre E=\bre \BQ_p$. We consider tuples $(X,\Fil^\bullet X, \iota)$ in the EL case, and $(X,\Fil^\bullet X,\iota,\lambda)$ in the PEL case, where the unfiltered data are the same as before, and 
\begin{itemize}
\item $\Fil^\bullet X$  is a filtration of $X$ by $O_B$-stable $p$-divisible subgroups
$$
0=  \Fil^0X\subset   \Fil^1V\subset \cdots \subset \Fil^\ell X=X.
 $$
\item In the PEL case, a {\em principal} polarization $\lambda: X\to X^\vee$ such  that the Rosati involution on $O_B$ coincides with $\ast$. (The requirement of the polarization being principal is due to our assumption on the unramifiedness of the integral datum.)
\end{itemize}
We impose the Kottwitz condition for every  $i=0,\cdots,\ell$,  
$$
 \charac \bigl(\iota(a);\Lie(\Fil^iX)\bigr)= \charac \bigl(a; (\Fil^iV)_0\bigr),\quad\forall a\in O_B,
$$ 
where $(\Fil^iV)_0$ is the weight zero subspace of $\Fil^iV$.  In the PEL case, we require that the filtration $\Fil^\bullet$ is {\em self-dual}, in the sense that, for every $i=0,\cdots,\ell$, we have an exact sequence (of fppf abelian sheaves) induced by the polarization $\lambda$:
$$
\xymatrix{  0 \ar[r]& 
\Fil^{\ell-i} X \ar[r]& X\simeq X^\vee\ar[r]&(\Fil^iX)^\vee\ar[r]& 0.
}
$$

Next we fix a {\em framing object} $(\BX,\Fil^\bullet\BX,\iota_\BX)$ over $\Spec\BF$  in the EL case (resp. a triple $(\BX,\Fil^\bullet\BX,\iota_\BX,\lambda_\BX)$ in the PEL case). We assume that the induced isocrystal with the filtration, with the $O_B$-action (and the alternating form induced by the polarization in the PEL case) is isomorphic to $(V\otimes_{\BQ_p}\bre\BQ_p, \Fil^\bullet V\otimes_{\BQ_p}\bre\BQ_p, b\sigma)$ (in the PEL case, $(V\otimes_{\BQ_p}\bre\BQ_p,  \Fil^\bullet V\otimes_{\BQ_p}\bre\BQ_p,(\cdot,\cdot),b\sigma)$ with the alternating form preserved up to a $\bre\BQ_p^\times$-factor).

We now state the definition of the moduli functor analogous to that of Rapoport--Zink, in the EL case:
$$
\xymatrix{\CM: \Nilp\ar[r]& Sets},
$$
which associates to $S\in\Nilp$ the set of isomorphism classes $(X,\Fil^\bullet X,\iota, \rho)$ where $(X,\Fil^\bullet X,\iota)$ is as above and $$
\xymatrix{\rho: X\times_S \ov S\ar[r]& \BX\times_{\Spec\BF}\ov S }
$$ is an $O_B$-linear quasi-isogeny that preserves the filtration $\Fil^\bullet X$ and $\Fil^\bullet \BX$. In the PEL case, $\CM(S)$ is the set of isomorphism classes $(X,\Fil^\bullet X,\iota, \lambda,\rho)$, where $(X,\Fil^\bullet X,\iota,\lambda)$ is as above and $$
\xymatrix{\rho: X\times_S \ov S\ar[r]& \BX\times_{\Spec\BF}\ov S }
$$
 is an $O_B$-linear quasi-isogeny that  preserves the filtration $\Fil^\bullet X$ and $\Fil^\bullet \BX$ and preserves the polarizations up to a factor in
$\BQ_p^\times$, locally on $\ov S$.

The following result follows essentially from the representability theorems in \cite[\S2,~\S3]{RZ}. Note that we have assumed that $\CD_{\BZ_p}$ is unramified.
\begin{theorem}\label{th:RZ fil}
The functor $\CM$ is (pro-)representable by a formal scheme, formally  locally of finite type, and formally smooth over $\Spf O_{\bre E}$.
\end{theorem}

\begin{proof}
We will defer the proof of the formal smoothness after we introduce the local model. 

We consider the EL case and the same argument applies to the PEL case. 
We consider the RZ moduli functors $\CM_i$ for the framing objects 
 $(\Fil^i\BX,\iota_{\Fil^i \BX})$, $i=1,\cdots,\ell$. In view of the representability theorem \ref{th:RZ} of Rapoport--Zink,  it suffices to consider the relatively representability of the natural map
 \begin{align}\label{eq:j}
 \xymatrix{j:\CM\ar[r]& \prod_{i=1}^\ell \CM_i. }
 \end{align}
 The map $j$ factorizes through $\CM^\bullet$ which parameterize similar data as $\CM$ except that $\Fil^iX\to \Fil^{i+1}X$ is not required to be  injective (as a morphism of fppf abelian sheaves). Then $\CM^\bullet$ is the locus where the quasi-homomorphisms $\Fil^i\BX\to \Fil^{i+1}\BX$ pull back to homomorphisms $\Fil^i X\to \Fil^{i+1} X$. By \cite[Prop.~2.9]{RZ}\footnote{Note that \cite[Prop.~2.9]{RZ} is stated only for quasi-isogenies; but the proof applies verbatim to  quasi-homomorphisms.},  $\CM^\bullet$ is a closed formal subscheme of $\prod_{i=1}^\ell \CM_i $.  Finally the injectivity imposes an open condition and hence  $\CM$ is an open formal subscheme of $\CM^\bullet$. The desired result follow.
 
\end{proof}

\begin{remark}Unlike the reductive case of RZ spaces, the irreducible components of the reduced scheme of $\CM$ are not necessarily proper.
The moduli functor  $\CM^\bullet$ may be viewed as a partial ``compactification" of $\CM$. 
\end{remark}
\begin{remark}Closely related to the above constructions, 
one may also consider
moduli functors in the {\em degenerate} PEL case, i.e., the alternative form on $V$ is degenerate. In fact, from the filtered rational RZ datum $\CD$ of PEL type, one has an induced datum from $\Fil^{\ell-1}V$, which carries a degenerate alternating form.
\end{remark}

\subsection{Connected components}
As a natural question, what is the set of connected components of $\CM$ defined in the last subsection? Similar questions for RZ spaces (and generalizations to local Shimua varieties in \cite{RV} with parahoric level structure) have received considerable attention in recent years. For the moduli space defined here, it seems unclear what the answer should be. Nevertheless we can partly answer the question in some special examples related to the unitary RZ spaces, where we have a very explicit description of the reduced scheme (via the BT-stratification), see \S\ref{ss:pi 0}.

\subsection{Local models}\label{ss:LM fil}Analogous to \S\ref{ss:LM},  the local model construction in \cite[\S3.26]{RZ} can be carried over to the fixed unramified integral filtered RZ data $\CD_{\BZ_p}$. We denote $\Lambda^i=\Fil^i\Lambda:=\Fil^iV\cap\Lambda$, an $O_B$-stable lattice. In the PEL case, the unramified hypothesis implies that there is a short exact sequence of $O_B$-modules
\begin{align}\label{eq:lat}
\xymatrix{0\ar[r]&\Lambda^{\ell-i} \ar[r]& \Lambda\simeq\Lambda^* \ar[r]&( \Lambda^{i})^*\ar[r]&0}.
\end{align}

We consider the following functor from the category $\Sch_{O_E}$ of schemes over $O_E$
$$
\xymatrix{\bM^{\loc}: \Sch_{O_E}\ar[r]& Sets},
$$which associates to $S\in \Sch_{O_E}$ the set of isomorphism classes of the following data:
\begin{itemize}
\item  A flag of locally free $\CO_B\otimes_{\BZ_p}\CO_S$-modules $\CF^\bullet$ on $S$: 
$$
0=\CF^0\subset \CF^1\subset \cdots\subset \CF^\ell=\CF.
$$
\item  Homomorphisms of  $\CO_B\otimes_{\BZ_p}\CO_S$-modules 
$$t_i:\Lambda^i\otimes_{\BZ_p} \CO_S\to \CF^i, \quad i=1,\cdots,\ell.
$$
\end{itemize}
We require the following conditions to hold:
\begin{itemize}
\item  all $\CF^i$ are subbundles of $\CF$ (i.e., $\CF/\CF^i$ are locally free), and the action of $O_B$ satisfies the analog of Kottwitz condition
$$
 \charac \bigl(a;\CF_i\bigr)= \charac \bigl(a; (\Fil^iV)_0\bigr),\quad\forall a\in O_B,
$$ 

\item The homomorphisms $t_i$ are all surjective and the following diagram commutes$$
\xymatrix{ \Lambda^1\otimes_{\BZ_p} \CO_S \ar@{->>}[d]^{t_1} \ar@{^(->}[r]&\cdots  \ar@{^(->}[r]&  \Lambda^i\otimes_{\BZ_p} \CO_S \ar@{->>}[d]^{t_i}   \ar@{^(->}[r] & \cdots\ar@{^(->}[r]   & \Lambda^\ell\otimes_{\BZ_p} \CO_S  \ar@{->>}[d]^{t_\ell}\\
\CF^1\ar@{^(->}[r]&\cdots \ar@{^(->}[r]&\CF^i \ar@{^(->}[r]&\cdots\ar@{^(->}[r]&\CF^\ell=\CF.
}
$$
In particular, when $i\leq j$, $\CF^i$ the image of $ \Lambda^i\otimes_{\BZ_p} \CO_S$ under $t_j$,  and hence $(t_i,\CF^i)$ is determined by $(t_j,\CF^j)$.
\item In the PEL case, the isomorphism $\Lambda^\ast\simeq \Lambda$ induces the following commutative diagram
$$\xymatrix{ \Lambda\otimes_{\BZ_p} \CO_S\ar[d]^{t_\ell}\ar[r]^\sim&(\Lambda \otimes_{\BZ_p} \CO_S)^*\ar[d] \\
\CF\ar[r]^\sim & (\ker(t_\ell))^*.}
$$
(Note that our polarization is assumed to be principal.) Here $\ast$ denotes the $\CO_S$-linear dual.
Together with \eqref{eq:lat} we have a commutative diagram for all $i=1,\cdots,\ell$:
\begin{align}
\xymatrix{ 0\ar[r]&\Lambda^{\ell-i}\otimes_{\BZ_p} \CO_S \ar[d] \ar[r]&\Lambda\otimes_{\BZ_p} \CO_S \simeq ( \Lambda \otimes_{\BZ_p}\CO_S)^*\ar[d]\ar[r]&( \Lambda^i\otimes_{\BZ_p} \CO_S )^* \ar[d]\ar[r]&0\\
 0\ar[r]&\CF^{\ell-i} \ar[r]&\CF\simeq  (\ker(t_\ell))^*\ar[r]&  (\ker(t_i))^* \ar[r]&0}
\end{align}
where the bottom row is required to be exact. 
\end{itemize}

The functor $\bM^{\loc}$ is represented by a quasi-projective scheme over $\Spec O_{ E}$. Unlike the reductive case in \cite{RZ} (cf. \S\ref{ss:LM}), it is {\em not necessarily} projective, as we will see from the proof of Lemma \ref{lem:LM fil} below.   

We consider the (unfiltered) RZ data $\gr^i\CD_{\BZ_p}$ associated to $ \gr^i V$, $1\leq i\leq \ell$ in the EL case and $1\leq i\leq [\frac{\ell+1}{2}]$ in the PEL case (cf. \S\ref{ss:fil RZ data}). Let $\bM^{\loc}_{i}$ be the corresponding local models as defined in \S\ref{ss:LM}. Then we have a natural morphism 
$$
\xymatrix{\bM^{\loc}\ar[r]&\prod_{i} \bM^{\loc}_{i}.}
$$
Note that each factor in the right hand side is the usual local model in \cite{RZ} and is smooth over $\Spec O_E$ under our assumption of unramifiedness of the datum $\CD_{\BZ_p}$.

\begin{lemma}\label{lem:LM fil}
The morphism $\bM^{\loc}\to\prod_{i} \bM^{\loc}_{i}$ is smooth and quasi-projective. In particular, 
the scheme $\bM^{\loc}$ is smooth over $\Spec O_{ E}$.
\end{lemma}
\begin{proof}
Similar to  $\bM^{\loc}$ we define a functor  $\bM^{\flat,\loc}$, only recording the first $\ell-1$ steps in $\CF^\bullet$ in the EL case, and the quotients $\CF^i/\CF^1, i=1,2,\cdots, \ell-1$  in the PEL case.
Then we have 
a natural morphism 
 \begin{align}\label{eq:pi}
\xymatrix{\pi: \bM^{\loc}\ar[r]&\bM^{\flat,\loc}\times \bM^{\loc}_{?}}
\end{align}
where $?=\ell$ in the EL case, and $?=1$ in the PEL case.

In the EL case, by induction on $\ell$, it suffices  to show that $\pi$ is representable, quasi-projective and smooth.  Consider the commutative diagram 
 \begin{align}\label{eq:ker}
\xymatrix{ 0\ar[r]&\Lambda^{\ell-1}\otimes_{\BZ_p} \CO_S \ar[d]^{t_{\ell-1}} \ar[r]&\Lambda\otimes_{\BZ_p} \CO_S \ar[d]^{t_\ell}\ar[r]&( \Lambda/\Lambda^{\ell-1})\otimes_{\BZ_p} \CO_S  \ar[d]^{t}\ar[r]&0\\
 0\ar[r]&\CF^{\ell-1} \ar[r]&\CF\ar[r]&  \CF/\CF^{\ell-1} \ar[r]&0}
\end{align}
By the snake lemma we obtain a short exact sequence of locally free $O_B\otimes_{\BZ_p}\CO_S$-modules:
 \begin{align*}
\xymatrix{ 0\ar[r]&\ker(t_{\ell-1})\ar[r]&\ker(t_\ell) \ar[r]&\ker(t)\ar[r]&0}.
\end{align*}
Moreover, we have an induced diagram of locally free $O_B\otimes_{\BZ_p}\CO_S$-modules
 \begin{align}\label{eq:def E}
\xymatrix{ 0\ar[r]&\frac{\Lambda^{\ell-1}\otimes_{\BZ_p}\CO_S } {\ker(t_{\ell-1})}\ar[d]^{=} \ar[r]&\CE\ar[r]\ar[d]&\ker(t)\ar[d] \ar[r]&0\\ 
0\ar[r]&\frac{\Lambda^{\ell-1}\otimes_{\BZ_p}\CO_S } {\ker(t_{\ell-1})} \ar[r]&\frac{ \Lambda\otimes_{\BZ_p} \CO_S }{\ker(t_{\ell-1})}\ar[r]&( \Lambda/\Lambda^{\ell-1})\otimes_{\BZ_p} \CO_S  \ar[r]&0}
\end{align}
where the top row is the pull-back of the bottom one via the right vertical map. 
Now, for any given $t_{\ell-1}$ and $t$,  the datum of $t_\ell$ is equivalent to the datum of a $O_B\otimes_{\BZ_p}\CO_S$-subbundle $\ker(t_\ell)/\ker(t_{\ell-1})$ of $\CE$, required to have trivial intersection with $\frac{\Lambda^{\ell-1}\otimes_{\BZ_p}\CO_S } {\ker(t_{\ell-1})}$. Note that the quotient bundle of $\CE$ by $\ker(t_\ell)/\ker(t_{\ell-1})$ is isomorphic to $\CF^{\ell-1}$, hence the characteristic polynomial for the action of $a\in O_B$ is given by ${\rm char}(a; (\Fil^{\ell-1}V)_0)$.

Locally (on  $\bM^{\flat,\loc}\times \bM^{\loc}_{\ell}$, considering the universal bundles) we may trivialize both $\CE$ and its subbundle $\CE_0:=\frac{\Lambda^{\ell-1}\otimes_{\BZ_p}\CO_S } {\ker(t_{\ell-1})}$. Then the fiber of $\pi$ is isomorphic to (an open subscheme of) the Hilbert/Quot scheme, denoted by ${\rm Hilb}_{\ker(t),(\Fil^{\ell-1}V)_0}$ (to indicate the action of $O_B$ on the kernel and the cokernel of the map $\lambda$ below),  parameterizing surjective homomorphisms of  $\CO_B\otimes_{\BZ_p}\CO_S$-modules
\begin{align}\label{eq:Quot}
\lambda: \CE\to \CL,
\end{align}
such that ${\rm char}(a;\CL)={\rm char}(a; (\Fil^{\ell-1}V)_0),\forall a\in O_B$, with the additional condition that 
$$
(\heart)\quad\lambda|_{\CE_0} \text{ is an isomorphism.}
$$  Without the condition $(\heart)$, the Hilbert scheme is a product of Grassmannians, hence projective and smooth over $\Spec O_E$. The condition $(\heart)$ is an open condition. Therefore the map $\pi$ is smooth and quasi-projective. Moreover, $\pi$ is projective precisely when $\rank\CF^{\ell-1}=0$ or $\rank\ker(t)=0$ (e.g., the two factors in the target of $\pi$ are in the definite case, see Example \ref{ex:def case}).

We now turn  to the PEL case. Note that $\bM^{\flat,\loc}$ is the local model for the filtered RZ datum induced by $\Fil^{\ell-1}/\Fil^1V$, hence also in the PEL case. By induction it suffices to show that the map \eqref{eq:pi} $\pi: \bM^{\loc}\to\bM^{\flat,\loc}\times \bM^{\loc}_{1}$ is smooth.  
Consider the analogous functor $\bN^{\loc}$, recording the first $\ell-1$ steps in $\CF^\bullet$. Then $\pi$ factors as
 \begin{align}\label{eq:pi PEL}
\xymatrix{ \bM^{\loc}\ar[r]^-{\pi_1}&\bN^{\loc}\ar[r]^-{\pi_2}& \bM^{\flat,\loc}\times \bM^{\loc}_{1}}
\end{align}
and it suffices to show that both $\pi_1$ and $\pi_2$ are  representable, quasi-projective and smooth.
The case for $\pi_2$ is similar to the EL case and its fiber is locally a Hilbert scheme ${\rm Hilb}_{(\Fil^{\ell-1}/\Fil^1V)_1,(\Fil^1V)_0}$, where $(\Fil^{\ell-1}/\Fil^1V)_1$ stipulates the action of $O_B$ on $\ker(t)$ for $t$ in the diagram:
 \begin{align*}
\xymatrix{ 0\ar[r]&\Lambda^{1}\otimes_{\BZ_p} \CO_S \ar[d]^{t_{1}} \ar[r]&\Lambda^{\ell-1}\otimes_{\BZ_p} \CO_S \ar[d]^{t_{\ell-1}}\ar[r]&( \Lambda^{\ell-1}/\Lambda^{1})\otimes_{\BZ_p} \CO_S  \ar[d]^{t}\ar[r]&0\\
 0\ar[r]&\CF^{1} \ar[r]&\CF^{\ell-1}\ar[r]&  \CF^{\ell-1}/\CF^1 \ar[r]&0.}
\end{align*}
For the morphism $\pi_1$, we denote by $ ( \ker(t_{i}) )^\perp$  the kernel of the composite map
 \begin{align}\label{eq:pi PEL}
\xymatrix{  \Lambda\otimes_{\BZ_p} \CO_S \ar[r]^\sim&(\Lambda \otimes_{\BZ_p} \CO_S)^*\ar[r] & \ker(t_{i})^\ast .}
\end{align}
Note that $( \ker(t_{\ell}) )^\perp= \ker(t_{\ell}) $ by the condition in the definition of local model. 
By \eqref{eq:ker}, $\ker(t_{\ell-1})$ is contained in $\ker(t_\ell)$ and hence  we have 
$$\ker(t_{\ell-1})\subset \ker(t_{\ell})= ( \ker(t_{\ell}) )^\perp\subset ( \ker(t_{\ell-1}) )^\perp.$$
Consider the quotient $\CE:=( \ker(t_{\ell-1}) )^\perp /\ker(t_{\ell-1})$ with the induced perfect alternating form. Then the datum of $t_\ell$ is equivalent to a subbundle  $\CL\subset \CE$, which is a Lagrangian  (i.e., the composite map $\CL\to \CE\to\CE^\ast\to\CL^*$ is zero), such that  $O_B$ acts on the quotient $\CE/\CL$ according to the action of $O_B$ on $ \ker(t_{\ell})/ \ker(t_{\ell-1})\simeq\CF_1^\ast$ and such that $\CL$ intersects trivially with $(\Lambda^{\ell-1}\otimes_{\BZ_p} \CO_S) /\ker(t_{\ell-1})$. Now it is easy to see that $\pi_2$ is  representable, quasi-projective and smooth. 
This completes the proof.
\end{proof}
Let $\hat\bM^{\loc}$ denote the $p$-adic completion of the base change $\bM^{\loc}\times_{\Spec O_{ E}}\Spec O_{ \bre E}$.
The Grothendieck--Messing deformation theory implies that any point of $\CM$ has an \'etale neighborhood which is formally \'etale  over $\hat\bM^{\loc}$  (the argument of \cite[\S3.32]{RZ} applies verbatim). Hence the smoothness claimed in Theorem \ref{th:RZ fil} follows from Lemma \ref{lem:LM fil}.

The proof of Lemma \ref{lem:LM fil} also tells us how to compute the relative dimension of $\bM^\loc$ and hence $\CM$.  In Example \ref{ex:EL fil}  (the EL case), 
the relative dimension of the map \eqref{eq:pi} is given by that of the Hilbert scheme ${\rm Hilb}_{\ker(t),(\Fil^{\ell-1}V)_0}$ (cf. \eqref{eq:dim EL}):
$$
\sum_{\varphi\in \Hom_{\BQ_p}(F,\ov\BQ_p)} \left(s_{\varphi}^{\ell}\sum_{i\leq \ell-1}r_{\varphi}^{i}\right).
$$
It follows that the relative dimension of $\CM$ over $\Spf O_{\bre E}$ is 
\begin{align}\label{eq:dim EL fil}
\sum_{\varphi\in \Hom_{\BQ_p}(F,\ov\BQ_p)} \sum_{j=1}^\ell\left( s_{\varphi}^{j}\sum_{1\leq i\leq j}r_{\varphi}^{i}\right).
\end{align}
In Example \ref{ex:PEL fil}  (the unitary PEL case), similarly we get the relative dimension of the map \eqref{eq:pi} as the sum of that of the map $\pi_1$ and $\pi_2$ in \eqref{eq:pi PEL}. The relative dimension of the map $\pi_2$ is 
$$
\sum_{\varphi\in \Hom_{\BQ_p}(F,\ov\BQ_p)} \left(r_{\varphi}^{1}\sum_{2\leq i\leq \ell-1}s_{\varphi}^{i}\right).
$$
By the total definiteness of $\gr^i, i\leq [(\ell-1)/2]$, the relative dimension of $\pi_1$ is zero.
\begin{example}
We specialize to the unitary PEL case: $F_0=\BQ_p$ and $\dim V=n$. Let $ \Hom_{\BQ_p}(F,\ov\BQ_p)=\{\varphi,\ov\varphi\} $.
Suppose 
\begin{itemize}
\item  
$(r_{\varphi},s_{\varphi})=(r,n-r)$. Then $(r_{\ov\varphi},s_{\ov\varphi})=(n-r,r)$.
\item $\ell=2j+1$ is odd, $\dim_F\gr^iV=1$ for all $i\leq j$. Then $\dim_F \gr^{j+1}V=n-2j$ 
\item $(r^i_{\varphi},s^i_{\varphi})=(0,1)$ for all $i\leq j$.  Then $(r^{j+1}_{\varphi},s^{j+1}_{\varphi})=(r,n-r-2j)$. (Implicitly $n\geq r+2j$.)
\end{itemize}
Note that $r^i_{\ov\varphi}=s^i_{\varphi}$. By the formula above we find the relative dimension of $\pi$ is equal to $r$ and hence the relative dimension of $\bM^\loc$ over $\Spec O_E$ is  
\begin{align}\label{eq:dim PEL fil}
r j +r(n-r-2j)=r(n-r)-rj.
\end{align}
As we will see in \S\ref{s:KR}, this is the same as the expected dimension of KR cycle attached to a totally isotropic rank-$j$ sublattice.
In fact the filtered RZ space in this case is the ``smooth locus" of the corresponding KR cycle, at least when $r=1$, cf. Prop.~\ref{p:fil2KR}.

 \end{example}
\section{Ginzburg--Rallis cycle}
\label{s:GR}
A particularly interesting example in the EL case arises from the Ginzburg--Rallis period \cite{GR}, which is (conjecturally) related to the exterior cube L-function on $\GL_6$. 

Consider the integral RZ datum $\CD_{G,\BZ_p}$ where  $F=\BQ_p, V=F^{\bigoplus 6}$ with a standard basis $\{e_1,\cdots, e_6\}$, $\mu=(0,0,0,1,1,1)$ (i.e., the cocharacter $z\mapsto \diag(1,1,1,z,z,z)$), $b$ basic, $\Lambda=\pair{e_1,\cdots,e_6}$. Henceforth we denote by  $\pair{e_1,\cdots,e_i}$ (resp.  $\pair{e_1,\cdots,e_i}_F$) the lattice (resp. the vector space) spanned by the specified vectors.   To formulate the moduli space, we fix the framing object 
$$
\BX=\BE\times \BE\times \BE,
$$
where $\BE$ is the (unique up-to-isomorphism) supersingular $p$-divisible group over $\BF$ of dimension one and height $2$. In terms of the notation in  Example \ref{ex:EL}, we have $r=s=3$. Then the RZ space $\CM_G$ is formally smooth 
 over $\Spf O_{\bre \BQ_p}$ of relative dimension $3\times 3=9$ by the formula \eqref{eq:dim EL}.
 
Consider the filtration 
$$
0\subset\pair{e_1,e_2}_F\subset \pair{e_1,e_2,e_3,e_4}_F\subset V 
$$
Its stabilizer is the parabolic subgroup $P$ with Levi subgroup isomorphic to $\GL_2^3$. 
Denote the filtered RZ datum by $\CD_{P,\BZ_p}$ and the RZ space by $\CM_{P}$. In terms of  Example \ref{ex:EL fil}, we have  $\ell=3$
and 
$$
r^i=s^i=1, \quad 1\leq  i\leq 3.
$$
Then the RZ space $\CM_P$ is formally smooth 
 over $\Spf O_{\bre \BQ_p}$ of relative dimension,  by the formula \eqref{eq:dim EL fil},
 $$
1\times 1+1\times 2+1\times 3=6.
 $$
 Note that $\CM_P$ comes with two natural morphisms 
 $$\xymatrix{ \CM_P \ar[d] \ar[r] &\CM_G \\
{\bf LT}\times_{\Spf O_{\bre\BQ_p}} {\bf LT}\times_{\Spf O_{\bre\BQ_p}} {\bf LT}&}
$$
where ${\bf LT}$ is the RZ space for $\GL_2,\mu=(0,1)$ and $b$ basis (i.e., each connected component is isomorphic to the Lubin--Tate deformation space in the height two case).

The Ginzburg--Rallis subgroup of $G=\GL_6$ is defined by 
$$
H=P\times_{\GL_2^3,\Delta}\GL_2,
$$
where $\Delta:\GL_2\to\GL_2^3$ is the diagonal embedding.
Then $H\simeq U\rtimes \GL_2$ where $U$ denotes the unipotent radical. We define the (local) Ginzburg--Rallis cycle $\CM_H$  as the fiber product 
 $$	\xymatrix{ \CM_H \ar[d] \ar[r] &\CM_P\ar[d] \\
{\bf LT} \ar[r]^-{\Delta} &\, {\bf LT}\times_{\Spf O_{\bre\BQ_p}} {\bf LT}\times_{\Spf O_{\bre\BQ_p}} {\bf LT}.}
$$
By Lemma \ref{lem:LM fil}, the vertical map is smooth of relative dimension $3$. In particular, we have the total dimension $\dim\CM_H=5 $, exactly half of the total dimension $\dim\CM_G=10$:
 $$	\xymatrix{i_{H,G}: \CM_H \ar[r] &\CM_G.}
$$
\begin{remark}
The drawback is that the morphism $i_{H,G} $ is not proper. One could consider a Laumon style compactification, but it is rather complicated and we hope to investigate its structure in the future. In fact, even the reduced scheme of $\CM_G$ seems not understood. \end{remark}
\begin{remark}There is a relative trace formula approach to the Ginzburg--Rallis period. We hope to formulate the arithmetic fundamental lemma conjecture in that context in the future.  
\end{remark}

\section{Connection to KR cycles}
\label{s:KR}
In this section we fix a $p$-adic field $F_0$ and an unramified quadratic extension $F$. Let $\varpi$ be a uniformizer of $F_0$ and $q$ the residue cardinality of $F_0$.
\subsection{The unitary RZ space}\label{ss:Nn}
Fix an integer $n \geq 1$. We explicate (a slight variant of) the RZ formal moduli scheme $\CN_n = \CN_{n, F/F_0}$ associated to (a special case of) the {\em unramified} integral RZ datum in the unitary PEL case, cf. Example \ref{ex:PEL}.  For simplicity  we will restrict ourselves to the case $F_0=\BQ_p$, though our notation will still be kept for the more general purpose.

For $S\in\Nilp$, we consider triples $(X, \iota, \lambda)$, where
 \begin{enumerate}
 \item[$\bullet$]
 $X$ is a $p$-divisible group of absolute height $2nd$ and dimension $n$ over $S$, where  $d:=[{F_0}: \BQ_p]$, 
 \item[$\bullet$]  $\iota$ is an action of $O_{F}$ such that the induced action of $O_{F_0}$ on $\Lie X$ is via the structure morphism $O_{F_0}\to \CO_S$,  and
 \item[$\bullet$] $\lambda$ is a principal polarization. 
 \end{enumerate}
  Hence $(X, \iota|_{O_{F_0}})$ is a formal $O_{F_0}$-module of relative height $2n$ and dimension $n$. We require that the Rosati involution $\Ros_\lambda$   on $O_{F}$ is the non-trivial Galois automorphism in $\Gal({F}/{F_0})$, and that the \emph{Kottwitz condition} of signature $(n-1,1)$ is satisfied, i.e.
\begin{equation}\label{kottwitzcond}
   \charac \bigl(a; \Lie X\bigr)=(T-a)^{n-1}(T-\ov a) \in \CO_S[T]
	\quad\text{for all}\quad
	a\in O_{F} . 
\end{equation} 
An isomorphism $(X, \iota, \lambda) \isoarrow (X', \iota', \lambda')$ between two such triples is an $O_{F}$-linear isomorphism $\varphi\colon X\isoarrow X'$ such that $\varphi^*(\lambda')=\lambda$.

Over the residue field $\BF$ of $O_{\breve {F}}$, we fix a {\em framing object}, i.e., a triple $(\BX, \iota_{\BX}, \lambda_{\BX})$ such that $\BX$ is supersingular, unique up to $O_F$-linear quasi-isogeny compatible with the polarization.  Then $\CN_n$ (pro-)represents the functor
$$
\xymatrix{\CN_n: \Nilp\ar[r]& Sets},
$$
which associates to each $S\in \Nilp$ the set of isomorphism classes of quadruples $(X, \iota, \lambda, \rho)$ over $S$, where the final entry is an $O_F$-linear quasi-isogeny of {\em height zero}\footnote{In \S\ref{ss:RZ red} on RZ space of PEL type, there was no restriction on the height of $\rho$. Hence here the space $\CN_n$ is a connected component of what's defined earlier.} 
\[
   \rho \colon X\times_S\ov S \to \BX \times_{\Spec \ov k} \ov S,
\]
such that $\rho^*((\lambda_{\BX})_{\ov S}) = \lambda_{\ov S}$.

In terms of Example \ref{ex:PEL}, we are in the case $B=F,\dim_FV=n$ and 
$$
(r_{\varphi},s_{\varphi})=(n-1,1), \quad(r_{\ov\varphi},s_{\ov\varphi})=(1,n-1),
$$ where $\varphi: F\incl \bre F$  is the fixed embedding. The formal scheme $\CN_n$ is connected, formally smooth over $\Spf O_{\breve {F}}$ of relative dimension $n-1$. 
When $n=1$,  the framing object is the supersingular formal $O_{F_0}$-module $(\ov\BE,\iota_{\ov\BE},\lambda_{\ov\BE})$ of (relative) height two and dimension one with signature $(0,1)$, and $\CN_1\simeq\Spf O_{\breve {F}} $ with the universal object being the canonical lifting $(\ov\sE, \iota_{\ov\sE},\lambda_{\ov\sE},\rho_{\ov\sE})$. Here the bar in the notation indicates the action of $O_F$ is through the Galois conjugate of the usual one.  We also consider the signature $(1,0)$ case with the framing object denoted by $(\BE,\iota_{\BE},\lambda_{\BE})$
and the universal object  $(\sE, \iota_{\sE},\lambda_{\sE},\rho_{\sE})$.

We make an explicit choice of the framing object $\BX=\BE^{n-1}\times\ov\BE$ with the obvious choice of $\iota_{\BX},\lambda_{\BX}$.
Let 
\begin{align}\label{eq:Vn}
\BV_n:=\Hom^\circ_{O_F}(\BE, \BX)
\end{align}
be the $F/F_0$-Hermitian space of {\em special quasi-homomorphisms}, where the Hermitian form is induced by the principal polarizations on $\BX$ and $\BE$. Note that $\BV$ is a non-split Hermitian space, i.e., it does not contain a self-dual lattice. 

There exists an isomorphism
\begin{equation}\label{Aut cong U}
 \xymatrix{  \Aut^\circ (\BX_n,\iota_{\BX_n},\lambda_{\BX_n}) \ar[r]^-{\sim}& \U\bigl(\BV_n\bigr)(F_0)}, 
\end{equation}
where the left-hand side is the group of quasi-isogenies of the framing object.
More concretely
$$
\U\bigl(\BV_n\bigr)(F_0)=\left\{g\in \End_F(\BV_n)\mid gg^\ast=\id\right\}.
$$
Here we denote $g^\ast:=\Ros_{\lambda_{\BX_n}}(g)$, the Rosati involution. 
Then the group $\U\bigl(\BV_n\bigr)(F_0)$ acts naturally on $\CN_n$ by changing the framing: 
$$g\cdot (X,\iota,\lambda,\rho) = (X,\iota,\lambda, g \circ \rho).
$$  

\begin{remark}\label{rem:t def}
We may also consider the totally definite case (cf. Example \ref{ex:def case}), i.e., the signature $(n,0)$ case. We may fix the framing object as $\BX'_n=\BE^n$ with the obvious choice of $\iota_{\BX'_n},\lambda_{\BX'_n}$. Then the resulting moduli space, denoted by $\CN'_n$,
is formally smooth of relative dimension zero.  Similarly define $\BV_n'=\Hom^\circ_{O_F}(\BE, \BX'_n)$, a split Hermitian space (hence isomorphic to $V$ in the RZ datum). Let $G=\U(\BV_n')$ and let $K$ the stabilizer of the lattice $\Hom_{O_F}(\BE, \BX'_n) $, which is a hyperspecial subgroup of $G$.  Then there is a $G(\BQ_p)$-equivariant isomorphism 
$$
 \xymatrix{ \CN'_n\ar[r]^-\sim &  \bigsqcup_{G(\BQ_p)/K}\Spf O_{\bre F}}.
$$
Here $G(\BQ_p)/K$ is naturally bijective to the set of self-dual lattices in  $\BV_n'$.

\end{remark}
\subsection{KR cycles}
\label{sec:kudla-rapop-cycl}

For any subset $L\subseteq \mathbb{V}_n$, we recall from \cite[Def.~3.2]{KR1} that the \emph{Kudla--Rapoport cycle} (or \emph{special cycle}) $\mathcal{Z}(L)$ is the closed formal subscheme of $\mathcal{N}_n$, which represents the functor sending each $S\in\Nilp$ to the set of isomorphism classes of tuples $(X, \iota, \lambda,\rho)\in\CN_n(S)$ such that for any $x\in L$, the quasi-homomorphism $$\rho^{-1}\circ x\circ \rho_{ \mathcal{E}}: \mathcal{E}_S\times_S\bar S\xrightarrow{\rho_{\mathcal{E}}} \BE\times_{\Spec\bar k}\bar S\xrightarrow{x}\mathbb{X}\times_{\Spec\bar k}\bar S\xrightarrow{\rho^{-1}} X\times_S \bar S$$  extends to a homomorphism $\mathcal{E}_S\rightarrow X$. Note that $\mathcal{Z}(L)$ only depends on the $O_F$-linear span of $L$ in $\mathbb{V}$. 

Note that when $L=\{x\}$ consists of a single non-zero element $x$, we also denote the special cycle by $\CZ(x)$. By \cite{KR1}, $\CZ(x)$ is a Cartier divisor on $\CN_n$, flat over $\Spf O_{\bre F}$.

\subsection{Filtered RZ space}\label{ss:fil U}
Let $m,r\in\BZ_{\geq 0}$ and $n=m+1+2r$.
We consider the filtered RZ datum in Example \ref{ex:PEL fil}, with the self-dual filtration 
$$
0=  \Fil^0V\subset   \Fil^1V\subset \cdots \subset \Fil^\ell V=V
$$
of length $\ell=2r+1$. We further assume that  
\begin{itemize}\item The graded quotients have dimensions:
$$\dim\gr^iV=\begin{cases} 1, & i\neq r+1,\\
n-2r, & i=r+1.
\end{cases}
$$
\item The signature $(r^i_\varphi,s^i_{\varphi})=(1,0)$ for all $i\leq r$. (Then the signature of $\gr^{r+1}V$ is $(n-1-2r,1)$.)
\end{itemize}
We denote by $P=P_{1^r,m+1,1^r}$ the parabolic subgroup of $G$ stabilizing the filtration.  

We explicate the formal moduli functor. 
 For $S\in\Nilp$, we consider tuples $(X, X_\bullet, \iota, \lambda)$, where $(X, \iota, \lambda)$ is as in \S\ref{ss:Nn}, and $$
0=X_0\subset X_1\subset\cdots\subset X_{2r+1}=X
$$
is a self-dual filtration (cf. \S\ref{ss:RZ fil}) such that the following Kottwitz condition holds: $\forall a\in O_{F} ,$ 
 $$
 \charac \bigl(a;\Lie(X_i)\bigr)=\begin{cases} (T-a)^{i}, &i\leq r, \\
 (T-a)^{n-1-j}(T-\ov a), &i=2r+1-j\geq r .
	\end{cases}
$$
Fix a framing object $(\BX, \BX_{\bullet}, \iota_{\BX}, \lambda_{\BX})$ over $\Spec \BF$. Then we define a functor
$$
\xymatrix{\CN_{P}: \Nilp\ar[r]& Sets},
$$
which associates to $S\in\Nilp$ the set of isomorphism classes $(X, X_\bullet,\iota,\lambda, \rho)$ where $(X,\Fil^\bullet X,\iota,\lambda)$ is as above and $$
\xymatrix{\rho: X\times_S \ov S\ar[r]& \BX\times_{\Spec\BF}\ov S }
$$ is an $O_F$-linear quasi-isogeny of {\em height zero} that preserves the filtration $X_\bullet$ and $ \BX_\bullet$, and such that $\rho^*((\lambda_{\BX})_{\ov S}) = \lambda_{\ov S}$.

Let $\CN_{1^r}$ denote the functor recording only the first $r$-steps $0=X_0\subset X_1\subset \cdots\subset X_r$.  Then we have a natural morphism
\begin{align}\label{eq:N P}
\xymatrix{\CN_{P}\ar[r]& \CN_{1^r}\times \CN_{n-2r}},
\end{align}
where the second factor sends $(X, X_\bullet,\iota,\lambda, \rho)$ to the $p$-divisible group $X_{r+1}/X_r$ with the induced additional structure.
Note that $\CN_{1^r}$ is an example of totally definite RZ space of EL type, and hence has relative dimension zero over $\Spf O_{\bre F}$.
A more explicit description of $\CN_{1^r}$  is as follows. Let $E_i=\Hom^\circ_{O_F}(\BE,\BX_i)\subset \BV$ and $E=E_r$. (A notational remark: henceforth the letter $E$ will no longer denote the reflex field of a local Shimura datum.) Then there is a natural bijection
\begin{align}\label{eq:N 1r}
\xymatrix{\CN_{1^r}\ar[r]^-{\sim}&\coprod_{\FL_{E}}\Spf O_{\bre F}},
\end{align}
where $\FL_{E}$ the set of complete flags $\CE_\bullet$ of lattices in $E=E_r$
$$
0\subset\CE_1\subset \CE_2\subset\cdots \subset \CE_r=\CE
$$
such that $\CE_i=\CE\cap E_i$ is of rank $i$. Upon choosing a basis\footnote{We need less than the basis; it suffices to fix a base point in the set $\{\CE_\bullet\}$. }, the index set is then isomorphic to $B_r(F)/B_r(O_F)$ where $B_r\subset \GL_{r,F}$ is the Borel subgroup of upper triangular matrices. Note that the set $\FL_{E}$ is also isomorphic to $\GL_r(F)/\GL_r(O_F)$ via the map sending the chain $\CE_\bullet$ to  the largest lattice $\CE\subset \BV_r$. We denote by  $\CN_{P,\CE_\bullet}$ the fiber of the map over the copy of $\Spf O_{\bre F}$ indexed by $\CE_\bullet\in\FL_{E}$:
$$\xymatrix{\CN_{P,\CE_\bullet}\ar[r]\ar[d]& \Spf O_{\bre F}\ar[d]^-{\CE_\bullet} \\ \CN_{P}\ar[r]&\CN_{1^r}}
$$

Similar to \eqref{eq:N P} we may further define a natural morphism, recording the graded quotients $X_{i}/X_{i-1}$  for $i=1,\cdots, r+1$ 
\begin{align}\label{eq:N MP}\xymatrix{\CN_{P}\ar[r]&\CN_{M}:= \CN_{1}\times_{ \Spf O_{\bre F}} \cdots \times_{\Spf O_{\bre F}}\CN_1\times_{\Spf O_{\bre F}} \CN_{m+1}}
\end{align}
where $M$ stands for the Levi of $P$, which is isomorphic to $\Res_{E/F}(\GL_1)^r\times G(W^\sharp)$.  Then the case $r=1$ of \eqref{eq:N 1r} asserts  (with the chosen basis)
\begin{align}\label{eq:N 11}
\xymatrix{\CN_{1}\ar[r]^-{\sim}&\coprod_{E^\times/O_E^\times\simeq\BZ}\,\, \Spf O_{\bre F}}.
\end{align}

\subsection{Connection to KR cycles}
Fix a flag of lattices $\CE_\bullet \in \FL_E$ as above, and let $\CE$ be the largest one in the flag, which is of rank $r$. Let $\CZ(\CE)^{\dag}$ the ``(formally) smooth locus" of the morphism $\CZ(\CE)\to \Spf O_{\bre F}$, i.e., 
the maximal open formal subscheme of $\CZ(\CE)$ which is flat over $\Spf O_{\bre F}$ and whose relative tangent spaces over $\Spf O_{\bre F}$ have a constant dimension (in this case $n-r-1$). 
\begin{proposition}\label{p:fil2KR}
There is a natural isomorphism:
$$
\xymatrix{ \CZ(\CE)^{\dag} \ar[r]^-{\sim}& \CN_{P,\CE_\bullet}.}
$$
\end{proposition}
\begin{proof}Let 
$(X, X_\bullet, \iota, \lambda)\in  \CN_{P,\CE_\bullet}(S)$. It is clear that   It suffices to show that the condition on the tangent space is equivalent to the condition that the homomorphism $\sE\otimes_{O_F}\CE\to X$ is a monomorphism, where $\sE\otimes_{O_F}\CE$ is Serre's tensor construction. 
Note that the condition for $\sE\otimes_{O_F}\CE\to X$  being a monomorphism can be checked on geometric points of $S$. Therefore it suffices to consider the case where $S=\Spec\BF$ is a geometric point $x$ on $\CN_n^\red$. Here $\BF$ is an arbitrary algebraically closed field in characteristic $p$.  

We recall some preliminary constructions (see \cite{VW}, \cite[\S2]{KR1}, \cite[\S3]{RTZ}, \cite[\S2]{LZ17} for more details). Let $\BD(X)$ be the (covariant) Dieudonn\'e module of $X$ with the induced action by 
$$
O_F\otimes_{O_{F_0}} O_{\bre F}\simeq  \prod_{\Hom_{O_{F_0}}(O_F, O_{\bre F})}O_{\bre F}.
$$ Let $A=\BD(X)_1$ be the eigenspace corresponding to the Galois conjugate of the tautological embedding $O_F\subset O_{\bre F}$. Via the framing $\rho$ we may identify the $\bre F$-vector space $A_{\bre F}:=A\otimes_{O_{\bre F}} \bre F$ with $\BV_{\bre F}:=\BV\otimes_{F}\bre F$,  respecting the $\sigma$-sesquilinear pairing on $A_{\bre F}$ induced by the polarization and the extension of the Hermitian paring on $\BV$ to $\BV_{\bre F}$. Then the lattice $A$ is a {\em special lattice} in $\BV_{\bre F}$, in the sense that $A^\vee\subset^1 A\subset \varpi^{-1} A^\vee$, where $A^\vee$ is the dual lattice. Then a point $x\in \CN_n^\red$ lies on $\CZ(\CE)$ if and only if the $\CE\subset A^\vee$. 

The map $\sE\otimes_{O_F}\CE\to X$ is a monomorphism if and only if the induced $O_{\bre F}$-linear map $\CE\otimes_{O_F} O_{\bre F}\to \BD(X)$ is saturated (i.e., the cokernel is torsion free). By the action on Lie algebra we may replace $ \BD(X)$ by $A=\BD(X)_1$. Furthermore, the saturation is equivalent to the condition that the $\BF$-linear map
\begin{align}\label{eq:mono}
\xymatrix{\CE\otimes_{O_F} \BF \ar[r]&A \otimes_{O_{\bre F}} \BF}
\end{align}
is injective.

By Grothendieck--Messing theory, the tangent space of $\CZ(\CE)\times_{\Spf O_{\bre F}}\Spec\BF$ at $x$ is isomorphic to the cokernel of the $\BF$-linear map (cf. \cite{LZ17})
\begin{align}\label{eq:tan}
\xymatrix{\CE\otimes_{O_F} \BF \ar[r]&A^\vee/\varpi A}
\end{align}
where the target $A^\vee/\varpi A$ is an $\BF$-subspace in $A/\varpi A=A \otimes_{O_{\bre F}} \BF$ of codimension one. Note that the KR cycle $\CZ(\CE)$ is the intersection of $r=\rank \CE$ Cartier divisors. The smoothness at $x$ is then equivalent to the above $\BF$-linear map \eqref{eq:tan} is injective, which is equivalent to the injectivity of the map \eqref{eq:mono}. This completes the proof.
\end{proof}

We now determine the reduced scheme of $\CZ(\CE)^{\dag}$.
\begin{proposition}\label{p:dag red}
Let $\CE$ be a totally isotropic lattice of rank $r$. Then
we have
$$
\CZ(\CE)^{\dag,\red} = \CZ(\CE)^{\red}\setminus \cup_{\CE\subsetneq \CE'\subset E}  \CZ(\CE')^{\red}
$$
and 
$$
\CZ(\CE)^{\dag} = \CZ(\CE)\setminus \cup_{\CE\subsetneq \CE'\subset E}  \CZ(\CE'),
$$
where the right hand side denotes the open formal subscheme of $\CZ(\CE)$ whose topological space is the open subscheme $\CZ(\CE)^{\red}\setminus \cup_{\CE\subsetneq \CE'\subset E}  \CZ(\CE')^{\red}$ of $\CZ(\CE)^{\dag,\red}$.

\end{proposition}
\begin{proof}
It suffices to prove the first assertion on the reduced schemes. We use the smooth criterion via the  injectivity of \eqref{eq:mono}. It is clear that $\CZ(\CE)^{\dag,\red} \subset \CZ(\CE)^{\red}\setminus \cup_{\CE'\supsetneq \CE}  \CZ(\CE')^{\red}.$ Now let $x\in \CZ(\CE)(\BF)$ and $A\subset \BV_{\bre F}$ the corresponding special lattice. Then $\CE\subset A^\vee$. 

It remains to show that, if the map \eqref{eq:mono} is not injective, then $x$ lies on $\CZ(\CE')^{\red}$ for some $\CE\subsetneq \CE'\subset E$. Let $\CE'=A\cap E$. The non-injectivity of \eqref{eq:mono} implies that $\CE\neq \CE'$. Since $E$ is totally isotropic, all vectors in $\CE'$ have integral Hermitian norms. By the claim below, we have $\CE\subset A^\vee$ and hence the point $x\in \CZ(\CE')$. 

{\em Claim: Let $u$ be a vector in $\BV$ such that $(u,u)\in O_F$. If $u\in A$, then $u\in A^\vee$.}

To show the claim, we denote by $\{\cdot,\cdot \}$ the $\sigma$-sesquilinear pairing on $\BV_{\bre F}$ and denote by $\Phi$ the automorphism $\id_{\BV}\otimes \sigma$.
 The dual lattice $A^\vee$ is characterized by $\{A^\vee, A\}=O_{\bre F}$ and we have $(A^\vee)^\vee=\Phi(A)$. It follows that $u=\Phi(u)\in \Phi(A)$ and hence we have $\{u, A^\vee \}\subset O_{\bre F}$.
  
Now suppose that $u\in A\setminus A^\vee$. Since $\dim_{\BF}A/A^\vee=1$, we must have
$$A=\pair{u}_{O_{\bre F}}+ A^\vee.$$
It follows from $(u,u)=\{u,u\}\in O_F$ that $$
\{u, A\}\subset  O_{\bre F}
$$
and hence $u\in A^\vee$. Contradiction!
\end{proof}

\begin{remark}
The assertion fails without the totally isotropic condition on $\CE$. For example, consider a rank one $\CE=\pair{u}$ with a generator $u$ such that its Hermitian norm $(u,u)=\varpi$. The special divisor $\CZ(\CE)=\CZ(u)$ is not formally smooth over $\Spf O_{\bre F}$ while $ \CZ(\CE')$ is empty for any lattice $\CE'\subset E$ such that $\CE'\supsetneq \CE$.
\end{remark}

\begin{example}[The case $\rank\CE=1$]Let $\CE=\pair{u}$ be a rank one lattice generated by a non-zero $u\in\BV$ with Hermitian norm $(u,u)=0$. Then by Prop.~\ref{p:dag red}, 
\begin{align}\label{eq:red r=1}
\CZ(\CE)^{\dag,\red}=\CZ(u)^\red\setminus \CZ(\varpi^{-1} u)^\red,
\end{align}
and 
\begin{align}
\CZ(\CE)^\dag=\CZ(u)\setminus \CZ(\varpi^{-1} u),
\end{align}
the restriction of $\CZ(u)$ to the open subscheme $\CZ(u)^\red\setminus \CZ(\varpi^{-1} u)^\red$ of $\CZ(u)^\red$. Note that $\CZ(\varpi ^{-1}u)$ is a Cartier divisor and is a closed formal subscheme of $\CZ(u)$ (though abstractly they are isomorphic!). The difference divisor $\CD(u)$ is defined as the difference of Cartier divisors 
\begin{align}
\CD(u):=\CZ(u)-\CZ(\varpi ^{-1} u),
\end{align} 
i.e., $\CD(u)$ is locally defined by $f_u/f_{\varpi ^{-1}u}=0$, where $f_u$ and $f_{\varpi ^{-1} u}$ are respectively the local equations defining $\CZ(u)$ and $\CZ(\varpi^{-1} u)$. Then $\CZ(\CE)^\dag$  is an open formal subscheme of $\CD(u)$ and hence we may view $\CD(u)$  as a compactification of $\CZ(\CE)^\dag$. It is analogous to the  Laumon style compactification of ${\rm Bun}_P\to {\rm Bun}_{\GL_n}$ for a parabolic $P$ of  $\GL_n$.

\end{example}
\subsection{The Bruhat--Tits stratification of $\CZ(\CE)^{\dag,\red}$}
We first recall the Bruhat--Tits stratification of the reduced scheme $\CN_n^\red$ of $\CN_n$, following the theorem of Vollaard--Wedhorn \cite{VW}\footnote{Our convention on vertex lattices slightly differ: the Hermtian form have integral values on our lattices $\Lambda$, which correspond to the dual lattices of those  while in {\em loc. cit.}.}, see also \cite[\S2.7]{LZ1}.

Let $\Ver(\mathbb{V})$ denote the set of  all vertex lattices $\Lambda\subseteq \mathbb{V}$, i.e., lattices (of full rank) such that $\Lambda\subset\Lambda^\vee\subset\varpi^{-1}\Lambda$. The reduced subscheme of $\mathcal{N}_n$ is a union of closed subschemes $\mathcal{V}(\Lambda)$
$$\mathcal{N}_n^\mathrm{red}=\bigcup_{\Lambda\in \Ver(\BV)} \mathcal{V}(\Lambda).
$$ Here $\CV(\Lambda)$ is the generalized Deligne--Lusztig variety \cite[\S2.5]{LZ1}, a smooth projective variety over $\BF$ of dimension $\frac{t(\Lambda)-1}{2}$. Here 
$$t(\Lambda):=\dim \Lambda^\vee/\Lambda$$ is called the type of $\Lambda$, which is an odd integer.
For two vertex lattices $\Lambda, \Lambda'$, we have $\mathcal{V}(\Lambda)\subseteq \mathcal{V}(\Lambda')$ if and only if $\Lambda\supseteq \Lambda'$; and $\mathcal{V}(\Lambda)\cap \mathcal{V}(\Lambda')$ is nonempty if and only if $\Lambda+\Lambda'$ is also a vertex lattice, in which case it is equal to $\mathcal{V}(\Lambda+\Lambda')$. We therefore have a \emph{Bruhat--Tits stratification} by locally closed subvarieties, 
$$
\mathcal{N}_n^\mathrm{red}=\bigsqcup_{\Lambda\in \Ver(\BV)}\mathcal{V}(\Lambda)^\circ, \quad \mathcal{V}(\Lambda)^\circ\coloneqq \mathcal{V}(\Lambda)\setminus\bigcup_{\Lambda\subsetneq \Lambda'}\mathcal{V}(\Lambda').
$$ The irreducible components of $\mathcal{N}_n^\mathrm{red}$ are exactly the projective varieties $\mathcal{V}(\Lambda)$, where $\Lambda$ runs over all vertex lattices of maximal type. The ``dual graph" of the stratification is the Bruhat--Tits building of the unitary group $\SU(\BV)$.

For $\CE\subseteq \mathbb{V}$ be an $O_F$-lattice of rank $r\ge1$. We define a subset of $ \Ver(\BV)$:
$$
 \Ver_\CE:=\{\Lambda\in \Ver(\BV)\mid\CE\subset \Lambda\}.
$$
By \cite[Prop.~4.1]{KR1}, the reduced subscheme $\mathcal{Z}(\CE)^\mathrm{red}$ of the KR cycle $\mathcal{Z}(\CE)$ is a union of Bruhat--Tits strata,
\begin{equation}
  \label{eq:KRstrat}
\mathcal{Z}(\CE)^\mathrm{red}=\bigcup_{\Lambda\in \Ver_\CE}\mathcal{V}(\Lambda)=\bigsqcup_{\Lambda\in \Ver_\CE }\CV(\Lambda)^\circ.  
\end{equation}
In fact, by \cite{LZ17}, we have $\mathcal{Z}(\Lambda)=\mathcal{V}(\Lambda)$ as closed formal subschemes of $\CN_n$.

\begin{remark}
In \cite{Va} Vandenbergen proved that $\CZ(\CE)^\red$ is connected when $\CE$ is {\em non-degenerate} (i.e. $E=\CE\otimes_{O_F}F$ is a non-degenerate Hermitian space). It is likely the method also proves the connectedness for any $\CE$, particularly totally isotropic $\CE$ for our interest in this paper.   
\end{remark}

Now we consider those $\Lambda\in\Ver_\CE$ such that $\CE$ is saturated in $\Lambda$: 
 $$
 \Ver^\dag_\CE:=\{\Lambda\in \Ver_\CE\mid\CE=E\cap \Lambda\}.
$$
For $\Lambda\in\Ver_\CE$, we define 
$$
\CV(\Lambda)^\dag=\CV(\Lambda)\setminus\bigcup_{\Lambda\subsetneq \Lambda',\Lambda'\notin  \Ver^\dag_\CE}\mathcal{V}(\Lambda').
$$
It is easy to see that 
$$
\CV(\Lambda)^\dag=\bigsqcup_{\Lambda\subsetneq \Lambda' \in    \Ver^\dag_\CE}\CV(\Lambda')^\circ.
$$
 Then by Prop.~\ref{p:dag red} and \eqref{eq:KRstrat},  we have a Bruhat--Tits stratification of $\CZ(\CE)^{\dag,\red}$.
\begin{align}\label{eq:red r=1}
\CZ(\CE)^{\dag,\red}=\bigcup_{\Lambda\in \Ver^\dag_\CE}\mathcal{V}(\Lambda)^\dag=\bigsqcup_{\Lambda\in \Ver^\dag_\CE }\CV(\Lambda)^\circ.
\end{align}

 Recall from \eqref{eq:j} that we have a morphism $\CZ(\CE)^\dag\to \CN_{n-2r}=\CN_{m+1}$.
Next we want to describe the induced morphism on their induced schemes. Denote $\BV^\flat=E^\perp/E$ with the induced Hermitian form. There is a map  
\begin{align}\label{eq:vert}
\xymatrix{\Ver_\CE^{\dag}\ar[r]& \Ver(\BV^\flat)} 
\end{align}
sending $\Lambda$ to $\Lambda^\flat=(E^\perp \cap \Lambda)/E\cap \Lambda=(E^\perp \cap \Lambda)/\CE$. By Lemma  \ref{lem: lat}, $\Lambda^\flat$ is a vertex lattice and the type does not increase: 
$$t(\Lambda^\flat)\leq t(\Lambda).
$$

Fix  a $\Lambda\in \Ver_\CE^{\dag}$ and its image $\Lambda^\flat\in  \Ver(\BV^\flat)$. Then we have a morphism of varieties over $\BF$
\begin{align}\label{eq:pi V}
\xymatrix{\pi_{\Lambda,\Lambda^\flat}:\CV(\Lambda)^{\dag}\ar[r]& \CV(\Lambda^\flat)}. 
\end{align}
One can describe it in terms of the special lattice $A\in \BV_{\bre F}$ but it is not clear what can be said about the map.
\begin{remark}
The reduced scheme $\CN_n^\red$ is a certain affine Deligne--Lusztig variety (in mixed characteristic). It is not clear (to the author) whether the reduced scheme $\CZ(\CE)^{\dag,\red}$ admits a similar description, and whether one can use such description to study the map \eqref{eq:pi V}. 
\end{remark}

To illustrate, we consider the rank one case $\CE=\pair{u}$. Given $\Lambda\in \Ver_\CE$, there are two cases
\begin{itemize}
\item $u\in  \varpi \Lambda^\vee$,
\item $u\notin \varpi \Lambda^\vee$.
\end{itemize}

In the former case, $(u,\Lambda)= \varpi O_F$, we have $t(\Lambda^\flat)=t(\Lambda)-2$. 
Then
$$
\CV(\Lambda)^{\dag}=\CV(\Lambda)\setminus \CV(\Lambda+\pair{u/\varpi}).
$$
The fiber of the map \eqref{eq:pi V} seems rather complicated.

In the latter case, $(u,\Lambda)=O_F$, we have $t(\Lambda^\flat)=t(\Lambda)$. The map \eqref{eq:pi V}  is an isomorphism. 

Note that, if $\Lambda$ has maximal type $n$ or $n-1$, then the latter case cannot happen.
\begin{remark}
The map $\pi_{\Lambda,\Lambda^\flat}$ behaves well with change of vertex lattices $\Lambda\subset \Lambda'$ in $\Ver_\CE$:  there is a natural commutative diagram
$$
\xymatrix{ \CV(\Lambda')^{\dag} \ar[d]\ar[r]& \CV(\Lambda'^\flat)\ar[d]\\
 \CV(\Lambda)^{\dag} \ar[r]& \CV(\Lambda^\flat)}
$$
where the two vertical maps are closed immersions.
\end{remark}

\subsection{Connected components: the example in the case $n=3$}
\label{ss:pi 0}

\begin{example}[Disconnectedness of $\CZ(u)^\dag\subset \CN_3$] Consider $n=3$ and a rank one lattice $\CE=\pair{u}$ spanned by a norm zero vector $u$. Then 
$$\Ver_{\pair{u}}^\dag=\bigl\{ \Lambda \in \Ver(\BV) \mid u\in \Lambda\setminus \varpi  \Lambda \bigr\}.$$
Let $\Ver_{\pair{u}}^{\dag, t=3}=\{\Lambda\in \Ver_{\pair{u}}^{\dag, t=3}\mid t(\Lambda)=3\}$. Fixing a base point $\Lambda_0$ in this set,  there is a natural bijection  
$$
\xymatrix{\Ver_{\pair{u}}^{\dag, t=3}\ar[r]^-{\sim} &N(F_0)/N(O_{F_0})}
$$
where $N$ is the unipotent radical of the parabolic subgroup corresponding to the flag $0\subset \pair{u}_F\subset \pair{u}_F^\perp\subset \BV$, and $N(O_{F_0})$ is defined with respect to $\Lambda_0$.

{\em Claim: $\CZ(u)^{\dag,\red}$ is disconnected of pure dimension one}.

It follows that
\begin{align}\label{eq:pi 0 n=3}
\CZ(u)^{\dag,\red}=\coprod_{ \Lambda\in \Ver_{\pair{u}}^{\dag, t=3}} \CV(\Lambda)^\dag.
\end{align}

To show the claim, first we note that there can not be zero dimensional connected components. In fact, for any type $1$ lattice $\Lambda\in \Ver_{\pair{u}}^\dag$, one can find a  type $3$ lattice $\Lambda'$ such that $u\in\Lambda'\subset \Lambda$: if $u\notin \varpi\Lambda^\vee$, then we can take $\Lambda'=\pair{u}+\varpi\Lambda^\vee$; if $u\in \varpi\Lambda^\vee$  then we can take any one of the $q+1$ lattices $\Lambda'\subset \Lambda$ of type $3$. 

It remains to show that the (punctured) DL curves $\CV(\Lambda)^\dag$ for $\Lambda\in \Ver_{\pair{u}}^{\dag, t=3}$ do not intersect. Suppose that $\Lambda_1,\Lambda_2\in \Ver_{\pair{u}}^{\dag, t=3}$ are adjacent, i.e., $\wt \Lambda=\Lambda_1+\Lambda_2$ is a vertex of type $1$.  We need to show that $\frac{1}{\varpi} u\in\wt \Lambda$, which implies that the intersection point $\CV(\Lambda_1)\cap\CV(\Lambda_2)=\CV(\wt\Lambda)$ is ``removed" in $\CZ(u)^{\dag,\red}$.

By $\Lambda_i^\vee=\varpi^{-1}\Lambda_i$ we have $\frac{1}{\varpi}u\in \Lambda_i^\vee$ and its image is nonzero in the quotient $\Lambda_i^\vee/\Lambda_i=\frac{1}{\varpi}\Lambda_i/\Lambda_i$. By
$$
\xymatrix{ \Lambda_i \ar@{^(->}[r]^{1}&\wt \Lambda \ar@{^(->}[r]^-{1}&\wt \Lambda^\vee \ar@{^(->}[r]^-{1} &\Lambda_i^\vee=\varpi^{-1}\Lambda_i .}
$$
and $\Lambda_1^\vee\cap  \Lambda_2^\vee= \wt\Lambda^\vee$, we see that $\frac{1}{\varpi} u\in \wt\Lambda^\vee$.  However, $(u,u)=0$ hence $\frac{1}{\varpi} u\mod \Lambda_1$ is isotropic in the 2-dimensional $k$-spase  $\wt \Lambda^\vee/\Lambda_1$ inside  $\Lambda_1^\vee/\Lambda_1$. Note that the $1$-dimensional line $ \Lambda^\vee/\Lambda_1$ inside $\wt \Lambda^\vee/\Lambda_1$ is characterized by $(x,x)=0$. Hence $ \frac{1}{\varpi} u\in \wt \Lambda$ as desired.

\end{example}

\begin{remark}Similar to \eqref{eq:pi 0 n=3} there is a description  of the set of connected components  $\CZ(u)^\dag$  when $n=4$. 
Is there a natural parameterization of the set of connected components of $\CZ(u)^\dag$ for an isotropic $u$ (more generally, $\CZ(\CE)^\dag$ for a totally isotropic $\CE$) when $n\geq 5$? 
\end{remark}

\section{The FL conjecture of Liu}\label{s:FL}

In this section we recall the fundamental lemma (FL) conjecture of Liu in the Bessel case \cite{Liu14}. This also sets up the analytic side that will be used to formulate our AFL conjecture.

\subsection{The Bessel subgroup}\label{ss:B sgp}
We recall \cite[\S12]{GGP} for the Bessel subgroups (in the unitary case). Let $F/F_0$ be a quadratic extension of fields.  Let $V$ be an $F/F_0$-Hermitian space of dimension $n \geq 2$, with the sesquilinear pairing denoted by $(\cdot,\cdot)$. Following {\it loc. cit.} we will denote by $G(V)$ the unitary group $\U(V)$. Let $E$ be a totally isotropic subspace of dimension $r$ so that we have a flag  $0\subset E\subset E^\perp\subset V$. (Here $E^\perp$ denotes the orthogonal complement of $E$ in $V$.) Let $W^\sharp= E^\perp/E$ be the quotient with the induced 
Hermitian structure:
\begin{align}\label{eqn:flag1}
\xymatrix{0\ar[r] & E\ar[r]&E^\perp\ar[r]& W^\sharp\ar[r]&0 .}
\end{align}
Fix a non-isotropic line $L\subset W^\sharp$.  Pulling-back the extension \eqref{eqn:flag1} along $L\incl W^\sharp$ we get an extension, denoted by $E^\sharp$,
\begin{align}\label{eqn:flag2}
\xymatrix{0\ar[r] & E\ar[r]&E^\sharp\ar[r]& L\ar[r]&0 .}
\end{align}
Let $W\subset W^\sharp$ be the orthogonal complement of $L$ in $W^\sharp$. Then we have $$
W^\sharp=W\oplus L.
$$

Now fix a  complete flag of (necessarily isotropic) subspaces in $E$:
$$
0=E_0\subset E_1\subset \cdots\subset E_r=E.
$$
It induces a (partial) flag of $V$ 
$$
0=V_0\subset V_1\subset \cdots\subset V_{2r+1}=V.
$$ 
by setting \begin{align}\label{eq:Vi}
V_{i}=\begin{cases}E_{i}, & 1\leq i\leq r,\\
E_{2r+1-i}^\perp, & r+1\leq i\leq 2r+1.
\end{cases}
\end{align}
Then we define a subgroup $H$ of $G(V)$ consisting of $g\in G(V)$ such that 
$$
\begin{cases}
gE_i\subset E_i,  g|_{E_i/E_{i-1}}=1,& i=1,\cdots, r,\\
g E^\sharp\subset E^\sharp, g|_{E^\sharp/E}=1.
\end{cases}
$$
Note that $g\in H$ also stabilizes $E_i^\perp$.
By the isomorphism $E^\perp/E^\sharp\simeq W$, we 
have a quotient map 
\begin{align}\label{eq:H pr}
\xymatrix{\fkq\colon H\ar@{->>}[r]&G(W)}
\end{align}
defined by $g\mapsto g|_{E^\perp/E^\sharp}$. (Note that all $g\in H$ automatically stabilizes $E^\perp$.)
The unipotent radical of $H$, denoted by $N$, is then the kernel of the quotient map $H\to G(W)$.

 \begin{remark}
Let $P=P_{1^r,m+1,1^r}$ be the parabolic subgroup of $G(V)$ which stabilizes a complete flag of (necessarily isotropic) subspaces in $E$. The Levi subgroup $M_P$ of $P$ is naturally isomorphic to $\Res_{F/F_0}(\GL_1)^r\times G(W^\sharp)$.  Then the group $H$ fits into a Cartesian diagram
\begin{align}\label{eq:H def2}
\xymatrix{H\ar[r] \ar[d]& P\ar[d]\\ G(W)\ar[r]& \Res_{F/F_0}(\GL_1)^r\times G(W^\sharp) .}
\end{align}
This provides a quick definition of $H$. However, it is then less obvious how to define the homomorphism \eqref{eq: def u} below.
\end{remark}

 \begin{remark}
In \cite{GGP} the authors start with an embedding $W\subset V$  and an orthogonal decomposition $
V=W\oplus W^\perp$ and $W^\perp=( E+E^\vee)\oplus L$. 
Then we have an isomorphism
\begin{align*}
H\simeq N\rtimes G(W),
\end{align*}
where $N$ is the unipotent radical of $P$ in \eqref{eq:H def2}, and $G(W)$ (now as a subgroup of $G(V)$) acts on $N$ by conjugation.  It seems more natural to treat $W$ as a subquotient rather than a subspace of $V$, as we will see in the moduli space of $p$-divisible group. 
\end{remark}

Let $N_{E^\sharp}$ denote the unipotent radical of the Borel subgroup $B_{E^\sharp}$ of $\Res_{F/F_0}\GL(E^\sharp)$ stabilizing the complete flag $
0=E_0\subset E_1\subset \cdots\subset E_r=E\subset E_{r+1}=E^\sharp.$ 
Then we have  a quotient map 
\begin{align}\label{eq: def u}
\xymatrix{u\colon H\ar@{->>}[r]&N_{E^\sharp}}
\end{align}
sending $g\in H$ to $g|_{E^\sharp}$.
Finally we fix a generic homomorphism (of algebraic groups)
\begin{align}
\xymatrix{\lambda\colon N_{E^\sharp}\ar[r]&F.}
\end{align}
Here by a ``generic" homomorphism we mean that  the stabilizer of the homomorphism under the natural action of the Levi torus $\Res_{F/F_0} (\GL_1)^{r+1}$ of $B_{E^\sharp}$ is as small as possible (in this case, the stabilizer  is exactly the center of $\Res_{F/F_0}\GL(E^\sharp)$). 

We now define an explicit homomorphism $\lambda$ by choosing  a basis $\{e_1,\cdots,e_r\}$ of $E$ such that 
$$
E_i=\pair{e_1,\cdots, e_i}_F, \quad 1\leq i\leq r,
$$
and a basis $\{e\}$ of $L$. We let $e_{r+1}\in E^\perp$ be a lifting of $e$. For $u\in N_{E^\sharp}$, we write $$u(e_{i+1})\equiv e_{i+1}+a_{i+1,i}\, e_{i}\mod E_{i-1},\quad a_{i+1,i}\in F_0, 1\leq i\leq r.$$
Then we define $\lambda$  by 
\begin{align}\label{eq:lambda H}
 \lambda(u)=\sum _{i=1}^{r}a_{i+1,i}.
 \end{align}
 For later use we also fix a ``dual basis" $\{e_1^\vee,\cdots, e_r^\vee\}$ (i.e. $\pair{e_i,e_j^\vee}=\delta_{ij}$) such that $e_{r+1}\perp e_i^\vee$ for all $i=1,\cdots, r$. Denote $E^\vee=\pair{e_1^\vee,\cdots,e_r^\vee}_F$. We then lift $W^\sharp$ to the unique subspace $(E+E^\vee)^\perp$ of $V$ and write $V=W^\sharp\oplus (E+E^\vee)$. In terms of the ``dual basis" we have 
  $$
 \lambda(u)=\sum _{i=1}^{r} (e_{i+1},e_i^\vee) .
 $$

 We view $H$ as a subgroup of $G=G(W)\times G(V)$ where the first factor is the natural projection \eqref{eq:H pr}. To rigidify the set up, in this paper let us simply assume that the discriminant of $L$ is trivial. Then the pair $(H,\lambda\circ u)$, up to $G$-conjugacy, depends only on the isomorphism class of  $V$ (or equivalently, only on  the isomorphism class of $W$), which we assume to contain a totally isotropic subspace of dimension $r$.  
 
We will be interested in the action of $H\times H$ on $G=G(W)\times G(V)$ by
$$
(h_1,h_2)\cdot (h,g):= (h_1hh_2,h_1gh_2).
$$
To simplify the action, we introduce the group
 $$
 \bH:= H\times_{G(W)} H=\{(h_1,h_2)\in H\times H\mid \fkq(h_1)=\fkq(h_2)\}.
 $$
 Then \eqref{eq:H pr} induces a surjective homomorphism  denoted by the same symbol
 \begin{align}\label{eq:qt H}
 \fkq: \bH\to G(W)
 \end{align} whose kernel is $N\times N$.
 It acts on $G(V)$ by 
 $$
(h_1,h_2)\cdot g:= h_1^{-1}gh_2.
$$
See \cite[\S4.2]{Liu14}. Then it is easy to see that the set of $(H\times H)(F_0)$-orbits in $G(F_0)$ is naturally bijective to the set of 
 $\bH(F_0)$-orbits in $G(V)(F_0)$. Henceforth we will freely switch between the two orbit spaces.

\subsection{Symmetric space}\label{s:gpthsetup}

Consider the symmetric space 
\begin{equation}\label{Sn def}
    S_n := \{\,\gamma\in \Res_{F/F_0}\GL_n\mid \gamma \ov \gamma=1_n\,\}.
\end{equation}
The group $\Res_{F/F_0}\GL_{n}$ acts on $S_n$  by
$$g\cdot \gamma = g^{-1} \gamma \ov  g.
$$

As before we assume $n=m+2r+1$.
Similar to the Bessel subgroup of the unitary group, we recall from \cite[\S2.1]{Liu14} the Bessel subgroup $H'$  of $\GL_{n,F}$. Let 
$$\{e_1,\cdots, e_n\}$$
be the standard basis of the n-dimensional vector space $V=F^{n}$. Let $E_i=\pair{e_1,\cdots,e_i}_F$ and $W=\pair{e_{r+2},\cdots, e_{n-r}}_F\subset W^\sharp=\pair{e_{r+1},\cdots, e_{n-r}}_F$. We view $\GL(W)$ as a subgroup $ \GL(W^\sharp)$ in the obvious way.   
Then the Bessel subgroup $H'$  of $\GL_{n,F}$ consists of $g\in \GL_{n,F}$ such that 
\begin{align}\label{eq:Bes GL}
\begin{cases}
gE_i\subset E_i,  g|_{E_i/E_{i-1}}=1,& i\in\{1,\cdots, r, n-r+1,\cdots, n\},\\
g|_{E_{m+1+r}/E_{r}}\in \GL(W),
\end{cases}
\end{align}
where we view $g|_{E_{m+1+r}/E_{r}}$ as an element in $\GL(W^\sharp)$ by the natural isomorphism $E_{m+1+r}/E_{r}\simeq W^\sharp$.
The kernel of the projection $H'\to \GL(W)\simeq \GL_{m,F}$ (the isomorphism defined by the above basis) is the unipotent radical denoted by $U_{1^r,m+1,1^r}$. Then 
$$
H'= U_{1^r,m+1,1^r}\rtimes\GL_{m,F}.
$$
We introduce the subgroup of $ \Res_{F/F_0}H'$
\[
   \bH' :=\Res_{F/F_0} U_{1^r,m+1,1^r}\rtimes\GL_{m,F_0} \subset \Res_{F/F_0}H'.
\]
with the induced action on $S_n$. In other words, $\bH'$ is defined by \eqref{eq:Bes GL} and further requiring that $g|_{E_{m+1+r}/E_{r}}\in \GL_{m,F_0}$. Denote the natural quotient map by
\begin{align}\label{eq:qt H'}
\xymatrix{\fkq'\colon \bH'\ar@{->>}[r]&\GL_{m,F_0} }.
\end{align}

Let $\wt E=\pair{e_1,\cdots, e_{r+1},e_{n-r+1},\cdots,e_n}_F$, naturally identified with $V/W$. Let $U_{1^{2r+1}}$ denote the unipotent radical of the Borel subgroup of $\GL(\wt E)\simeq \GL_{2r+1}$ with respect to the complete flag $0\subset \pair{e_1}_F\subset\cdots \subset \pair{e_1,\cdots, e_{r+1},e_{n-r+1},\cdots,e_n}_F$.
We have a surjective homomorphism
\begin{align}
\xymatrix{u'\colon \bH'\ar@{->>}[r]&U_{1^{2r+1}}},
\end{align}
sending $g\in  \bH'$ to the composition $$\xymatrix{\wt E\ar[r]& V\ar[r]^{g}&  V\ar[r]& V/W \simeq \wt E}.$$
(It is easy to verify that the composition indeed lies in the subgroup $U_{1^{2r+1}}$ of $\GL(\wt E)$.) Similar to \eqref{eq:lambda H}, we define a generic homomorphism,
\begin{align}
\xymatrix{\lambda'\colon U_{1^{2r+1} }\ar[r]&F}
\end{align}using the above basis $\{e_1,\cdots, e_{r+1},e_{n-r+1},\cdots,e_n\}$.

\subsection{Regular orbit matching}\label{ss: orb match}
We now consider the space of (regular) orbits for the two group actions introduced above: 
\begin{altenumerate} 
\item The action of $\bH$ on $G(V)$.
\item The action of $\bH'$ on $S_n$.
\end{altenumerate} 
Recall from \cite[Def.~4.5]{Liu14} that an element  $\gamma \in S_n(F_0)$ (resp. $g\in G(V)(F_0)$) is called {\em pre-regular} if its stabilizer under the action of
$\Res_{F/F_0}U_{1^r,m+1,1^r}$ (resp. $N\times N$) is trivial. By \cite[Lem.~4.6]{Liu14}, a $\Res_{F/F_0}U_{1^r,m+1,1^r}$-orbit of a pre-regular element contains a unique element of the form 
\begin{align}\label{eq:n form S}
\gamma=
\begin{bmatrix}  & &&&&& t'_1\\
& &&&&\udots&\\
& &&&t'_r&& \\
& &&\gamma^\sharp&&& \\
& &(\ov t'_r)^{-1}&&&& \\
&\udots &&&&& \\
(\ov t_1')^{-1}& &&&&& \\
 \end{bmatrix}, \quad t'_i\in F^\times,\,  \gamma^\sharp\in S_{m+1}(F_0).
\end{align}
See \cite[Rem.~4.8]{Liu14} for a more ``intrinsic" definition of the invariants  $t_i$.
 Such a form is called the {\em normal form} (of any element in the orbit). 

For the unitary groups, we take the fixed basis of $E$ and $E^\vee$.  By \cite[Lem.~4.7]{Liu14}, a $N\times N$-orbit of a pre-regular element contains a unique element of the form 
\begin{align}\label{eq:n form U}
g=
\begin{bmatrix}  & &&&&& t_1\\
& &&&&\udots&\\
& &&&t_r&& \\
& &&g^\sharp&&& \\
& &\ov t_r^{-1}&&&& \\
&\udots &&&&& \\
\ov t_1^{-1}& &&&&& \\
 \end{bmatrix}, \quad t_i\in F^\times,\, g^\sharp\in G(W^\sharp)(F_0),
\end{align}
which is called the {\em normal form} (of any element in the orbit).

Next we recall that the notion of {\em regular} (called regular semisimple in \cite{Z21}) elements in $S_{m+1}$ and $G(W^\sharp)$ (relative to the action of $\GL_m$ and $G(W)$ respectively).  Using the basis of $W$ and $L$ we identify $\End_F(W\oplus L)\simeq  {\rm Mat}_{m+1,F}$. Then an element $\xi\in  {\rm Mat}_{m+1,F}$  is {\em regular} if $\{e, \xi e,\cdots, \xi^m e\}$ is an $F$-basis of $W\oplus L$ and $\{e^\vee, e^\vee \xi ,\cdots, e^\vee \xi^m \}$ is an $F$-basis of $W^\vee\oplus L^\vee$, where $e^\vee$ is a basis of $L^\vee=\Hom_F(L,F)$. Then we say that $g^\sharp\in G(W^\sharp)(F_0)$ (resp. $\gamma^\sharp\in S_{m+1}(F_0)$) is {\em regular} if it is regular as an element of ${\rm Mat}_{m+1,F}$. A regular $g^\sharp\in G(W^\sharp)(F_0)$ and $\gamma^\sharp\in S_{m+1}(F_0)$ are said to match if they are $\GL(W)\simeq \GL_m(F)$-conjugate when they are viewed as elements in $ \GL_{m+1}(F)$. 

Finally, we say that an element $\gamma\in S_n(F_0)$ (resp. $g\in G(V)$) is {\em regular} if it is pre-regular and $\gamma^\sharp\in S_{m+1}(F_0)$ (resp. $g^\sharp\in G(W^\sharp)(F_0)$) in its normal form  is regular. Denote by $S_n(F_0)_\rs$ (resp. $G(V)(F_0)_\rs$) the subset of regular elements. Then $\bH'(F_0)$ (resp.  $\bH(F_0)$) preserves $S_n(F_0)_\rs$ (resp. $G(V)(F_0)_\rs$). We will denote by $\bigr[S_n(F_0)\bigr]_\rs$ (resp. $\bigl[\U(V) (F_0)\bigr]_\rs$) the set of regular orbits.

A regular $g\in G(V)(F_0)$ and $\gamma \in S_{n}(F_0)$ are said to match if their normal forms  \eqref{eq:n form S} and \eqref{eq:n form U} satisfy
\begin{enumerate}
\item $t_i=t_i'$ for all $i=1,\cdots, r$, and 
\item  $g^\sharp\in G(W^\sharp)(F_0)$ and $\gamma^\sharp\in S_{m+1}(F_0)$ match.
\end{enumerate}

From now on we further assume that $m\geq 1$.
By \cite[Prop. 4.12]{Liu14} the matching relation defines  a natural bijection of regular orbits, 
\begin{align}\label{eq:orb mat 0} 
\xymatrix{
 \coprod_{V} \bigl[\U(V) (F_0)\bigr]_\rs  \ar[r]^-\sim& \bigr[S_n(F_0)\bigr]_\rs},
\end{align}
 where the disjoint union runs over the set of isometry classes of $n$-dimensional $F/F_0$-Hermitian spaces $V$ containing a totally isotropic subspace of dimension $r$.

\subsection{Orbital integral: smooth transfer}\label{ss: transfer}
We recall orbital integrals \cite[\S4]{Liu14}. Now let $F/F_0$ be a quadratic extension of $p$-adic fields. 
Then there are exactly two isometry classes of $F/F_0$-Hermitian spaces of dimension $n$, denoted by $V_0$ and $V_1$. When $F/F_0$ is unramified, we will assume that $V_0$ has a self-dual lattice. Then the orbit bijection \eqref{eq:orb mat 0} becomes 
\[ \xymatrix{ \bigl[( \U(V_0) (F_0)\bigr]_\rs 
 \coprod\bigl[( \U(V_1) (F_0)\bigr]_\rs  \ar[r]^-\sim& \bigr[S_n(F_0)\bigr]_\rs}.
\]

Fix  a non-trivial continuous character $\psi_0: F_0\to \BC^\times$ , and set $\psi =\psi_0\circ\tr: F\to \BC^\times$. Abusing notation  we also denote the character 
 \begin{align}\label{eq:ch N}
\xymatrix{ \psi\colon N_{E^\sharp}(F)\ar[r]&\BC^\times},
 \end{align}
 sending $u\in N_{E^\sharp}(F)$ to  $\psi\circ \lambda( u)$, 
 \begin{align}
\xymatrix{ \psi\colon H(F_0)\ar[r]&\BC^\times}
 \end{align}
 sending $h\in H(F_0)$ to $\psi\circ \lambda\circ u(h)$, and 
  \begin{align}
\xymatrix{ \psi\colon \bH(F_0)\ar[r]&\BC^\times}
 \end{align}
  sending $(h_1,h_2)\in \bH(F_0)$ to $\psi(h_1^{-1})\psi(h_2)$. Similarly we denote
 \begin{align}
\xymatrix{ \psi\colon \bH'(F_0)\ar[r]&\BC^\times}
 \end{align}
 sending $\in \bH'(F_0)\mapsto \psi\circ \lambda\circ u(h)$.

 Let 
$$\xymatrix{\eta=\eta_{F/F_0}\colon F_0^\times\ar[r]&\{\pm 1\}}$$ 
be the quadratic character associated to $F/F_0$ by local class field theory.  Abusing notation  we also denote the character 
 \begin{align}
\xymatrix{ \eta\colon \bH'(F_0)\ar[r]&\{\pm 1\}}
 \end{align}
 sending $h\in \bH'(F_0)$ to $\eta\circ\det\circ \fkq'(h)$. Similarly, we let 
 \begin{align}
\xymatrix{ | \cdot|\colon \bH'(F_0)\ar[r]&\BR^\times}
 \end{align}
denote the character $h\mapsto | \det\circ \fkq'(h)|_{F_0}$.

For $V=V_0$ or $V_1$, $g \in G(V)(F_0)_\rs$ and $f \in C_c^\infty(G(V)(F_0))$, we define  the orbital integral
\[
   \Orb(g,f) := \int_{\bH(F_0)} f(h \cdot g)\psi(h) \, dh.
\]

For $\gamma \in S_n(F_0)_\rs$, $f' \in C_c^\infty(S_n(F_0))$, and $s \in \BC$, we define the orbital integral
\begin{equation}\label{Orb(gamma,f',s)}
   \Orb(\gamma,f',s) := \int_{\bH'(F_0)} f'(h\cdot \gamma ) \psi(h) \eta(h)| h|^s\, dh,
\end{equation}
where $\lv\phantom{a}\rv$ denotes the normalized absolute value on $F_0$.
We define the special values
\begin{equation}\label{Orb(gamma,f',s=0)}
   \Orb(\gamma, f') :=\omega(\gamma) \Orb(\gamma, f', 0)
	\quad\text{and}\quad
	\del(\gamma, f') := \omega(\gamma)\,\frac{d}{ds} \Big|_{s=0} \Orb({\gamma},  f',s),
\end{equation}
where $\omega(\gamma)$ is a transfer factor (defined in \cite[\S4.4,~(4.17)]{Liu14} at least when $F/F_0$ is unramified and $\gamma$ is a normal form).

\begin{definition}
A function $f' \in  C_c^\infty(S_n(F_0))$ and a pair of functions $(f_0,f_1) \in  C_c^\infty(\U(V_0)(F_0)) \times  C_c^\infty(\U(V_1)(F_0))$ are (smooth) \emph{transfers} of each other if for each $i \in \{0,1\}$ and each $g \in \U(V_i)(F_0)_\rs$,
\[
   \Orb(g,f_i) =  \Orb(\gamma,f')
\]
where $\gamma \in S(F_0)_\rs$ matches $g$.
\end{definition}

It is expected that for any given   $(f_0,f_1) \in  C_c^\infty(\U(V_0)(F_0)) \times  C_c^\infty(\U(V_1)(F_0))$ or $f' \in  C_c^\infty(S_n(F_0))$, a smooth transfer always exists, though this remains an open problem as of today.

\begin{remark}
We have assumed $m\geq 1$. The case $m=0$ (so that $n=2r+1$) requires little modification: in \eqref{eq:orb mat 0} there is only one Hermitian space $V$, i.e., the split one. 
\end{remark}

\subsection{Liu's FL conjecture}\label{s:FL}
We review the FL conjecture, cf.~\cite[\S4.4]{Liu14}. Now assume that $F/F_0$ is an unramified quadratic extension of $p$-adic field for an {\em odd} prime $p$. We assume that the character $\psi_0: F_0\to \BC^\times$ is of level zero (i.e., $\psi|_{O_{F_0}}=1$  but $\psi|_{\varpi^{-1}O_{F_0}}\neq 1$). Fix a self-dual $O_F$-lattice $\Lambda_0 \subset V_0$. We denote its stabilizer by $K_0\subset G(V_0)(F_0)$, a hyperspecial maximal compact open subgroup. Let $\bH'(O_{F_0})=K_0\cap \bH'(F_0)$. We normalize the Haar measure on 
$ \bH(F_0)$ (resp.  $\bH'(F_0)$) 
such that $\vol( \bH(O_{F_0}))=1$ (resp.  $\vol(\bH'(O_{F_0}))=1$). We have \cite[\S4.4]{Liu14}:
\begin{conjecture}[Liu's Fundamental Lemma conjecture]\label{FLconj}
The characteristic function $$\mathbf{1}_{S_n(O_{F_0})} \in C_c^\infty(S_n(F_0))$$ transfers to the pair of functions $$(\mathbf{1}_{K_0},0) \in C_c^\infty(G(V_0)(F_0))\times C_c^\infty(G(V_1)(F_0)).$$
\end{conjecture}

When $r=0$ this is the Jacquet--Rallis Fundamental Lemma conjecture. When $r>0$, the conjecture remains an open problem as of today.

\begin{remark}
 Liu also allows the case $m=0$ in \cite[\S4.4]{Liu14},  which was already formulated by Jacquet and proved in some low dimensional cases in \cite{J92}. For the formulation of the AFL conjecture, we will always assume $m\geq 1$.
 \end{remark}

The easy part of the conjecture is known by \cite[Prop.~4.16]{Liu14}, i.e., when $\gamma$ matches an element $g\in G(V_1)(F_0)$:
$$
 \Orb(\gamma,f')=0.
$$
In the next section, we will formulate a conjecture relating the first derivative $\del(\gamma, f')$ to certain arithmetic intersection numbers.
\subsection{FL: an interpretation as lattices counting}
\label{ss:lat}

Suppose now that $V$ is the split Hermitian space and that the special vector $e\in W^\sharp$ has norm $(e,e)=1$ (or a unit). In this subsection we let  $\CN'_{n}$ denote the set $\Ver^0(V)$ of self-dual lattices $\Lambda\subset V$. As the notation suggests, $\CN'_{n}$ is the set of connected components of the RZ space in the totally definite case, cf. Remark \ref{rem:t def}.

Recall we have fixed a basis $\{e_1,\cdots, e_r\}$ of $E$, a basis $\{e\}$ of $L$, and $e_{r+1}\in E^\sharp$ a lifting $e\in L$. 
We consider the $O_F$-lattices 
$$
\CE^0_i=\pair{e_1,\cdots, e_i}\subset E_i, \quad 1\leq i\leq r+1.
$$
Denote
\begin{align}\label{eq:CE0}
\CE^{0,\flat}=\CE^0_r\subset E,\quad \CE^{0}=\CE^0_{r+1}\subset E^\sharp.
\end{align}

Let $\CN'_{P}$ be the set of $0=\Lambda_0\subset \Lambda_1\subset \cdots \subset \Lambda_{2r+1}$ such that $\Lambda_{2r+1}\in \CN_n'$ and $\Lambda_i=V_i\cap \Lambda_{2r+1}$ for all $i=1,\cdots, 2r+1$. (Here recall $V_i$ from \eqref{eq:Vi}.) Then the map $\Lambda_\bullet\mapsto \Lambda_{2r+1}$ defines a bijection:
$$
\xymatrix{\CN'_{P}\ar[r]^-{\sim}&\CN_n'.}
$$
Therefore we will freely switch between the two sets. Let $\CN'_M$ denotes the set of tuples $(\CL_1,\CL_2,\cdots,\CL_m,\Lambda^\sharp)$ where $\CL_i$ are rank one lattices in  $V_i/V_{i-1}=E_{i}/E_{i-1}$ for $i=1,\cdots, r$ and $\Lambda^\sharp$ is a self-dual lattice in $W^\sharp$.
There is a natural map 
$$
\xymatrix{\CN'_{P}\ar[r]&\CN'_M.}
$$
sending $\Lambda_\bullet$ to $\CL_i=\Lambda_{i}/\Lambda_{i-1} $ for $i=1,\cdots, r$ and $\Lambda^\sharp=\Lambda_{r+1}/\Lambda_r$ (this is self-dual by Lemma \ref{lem: lat}).

Let  $\CN'_{m}=\Ver^0(W)$ (resp. $\CN'_{m+1}=\Ver^0(W^\sharp)$) denote the set of self-dual lattices in $W$  (resp. in $W^\sharp$). We have an embedding 
\begin{align}\label{eq:iota' m}
\xymatrix{\iota:\CN'_{m}\ar[r]&\CN'_{m+1}}
\end{align}
with the image consisting of lattices $\Lambda^\sharp\in \Ver^0(W^\sharp)$ such that $e\in \Lambda^\sharp$. There is an embedding
\begin{align}\label{eq:iota' m+1}
\xymatrix{\iota_{\{e_1,\cdots,e_r\}}:\CN'_{m+1}\ar[r]&\CN'_{M}}
\end{align}
with the image consisting of $(\CL_1,\cdots,\CL_r,\Lambda^\sharp)$ such that $\CL_i=\pair{e_i}$.  For the Bessel subgroup $H$, we define 
\begin{align}\label{eq:N' H}
\xymatrix{\CN'_{H}\ar[r] \ar[d]& \CN'_P\ar[d]\\   \CN'_{m}\ar[r]&\CN'_M}
\end{align}
where the bottom map is the composition of the two embeddings \eqref{eq:iota' m} and \eqref{eq:iota' m+1}.

Unfolding the definition we see that  $\CN'_{H}$ is the set of lattices  $\Lambda\in \CN_n'=\Ver^0(V)$ such that, for $1\leq i\leq r+1$, setting $\CE_i:=E_i\cap \Lambda$, then we have $\CE_{i}/\CE_{i-1}=\pair{e_i}$ (as lattices in $E_{i}/E_{i-1}$). (Here we warn the reader that, in \eqref{eq:Vi}, $V_i=E_i$ for $i\leq r$ but $V_{r+1}\neq E_{r+1}=E^\sharp$. This is the reason that here we wrtite $\CE_i$ rather than $\Lambda_i$.) We choose the more laborious definition above to be consistent with the arithmetic analog in the next section.

Let $\CN'_{m,n}= \CN'_{m}\times \CN'_{n}$ with the natural action by the product group $(G(W)\times G(V))(F_0)$. We then have a map
$$
\xymatrix{\CN'_H\ar[r]&\CN'_{m,n}}
$$
sending $\Lambda\in \CN'_H\mapsto (\Lambda^\flat,\Lambda)$.  (Here via the embedding $\CN'_H\incl\CN'_n$ we write an element in $\CN'_H$ as $\Lambda\in \CN_n'$.)

Let 
$\FL^\square_{E^\sharp}$ denote the set of (framed) complete flags $\CE_\bullet$ of lattices in $E^\sharp$:
$$
0\subset\CE_1\subset \CE_2\subset\cdots \subset \CE_{r+1}=\CE
$$
such that $\CE_i=\CE\cap E_i$ is of rank $i$ and
$\CE_{i}/\CE_{i-1}=\pair{e_i}, 1\leq i\leq r+1.$ (The word``framed" refers to $\CE_{i}/\CE_{i-1}=\pair{e_i}$ for the fixed $e_i$.) Then there is a natural bijection
\begin{align}\label{eq:fl sq'}
\FL^\square_{E^\sharp}\simeq N_{E^\sharp}(F)/N_{E^\sharp}(O_F),
\end{align}
where the neutral point in the right hand side corresponds to the (framed) flag  $\CE^0_\bullet$ that we have fixed via \eqref{eq:CE0}.  (Here $ N_{E^\sharp}$ is defined in \S\ref{ss:B sgp}.) The character of $N_{E^\sharp}(F)$ defined by \eqref{eq:ch N} induces a character
$$\xymatrix{\psi:\FL^\square_{E^\sharp}\ar[r] &\BC^\times.} 
$$
There is a natural map 
$$\xymatrix{ \CN'_H\ar[r] &\FL^\square_{E^\sharp}} 
$$
sending $\Lambda$ to $\CE_i=\Lambda\cap E_i, 1\leq i\leq r$ and $\CE=\Lambda\cap E^\sharp$.
We also denote by 
$$\psi:\CN'_H\to \BC^\times $$
 the pullback of the character of $\FL^\square_{E^\sharp}$. 

Then we may interpret the orbital integral in \S\ref{ss: transfer} as a weighted lattice counting: for a regular $g\in G(V)$,
\begin{align}\label{eq:orb 1U}
\Orb(g,1_{G(O_{F_0})})=\sum_{(\Lambda,\Lambda')\in \CN'_{H}\times \CN'_H \atop
\Lambda'=g\Lambda,\, \Lambda^\flat=\Lambda'^\flat}  \psi(\Lambda)\ov\psi(\Lambda').
\end{align}
 
As an analog of the intersection problem to be introduced in the next section, we show that the sum is finite.
\begin{proposition}\label{prop:int fin}
Suppose that $g\in G(V)(F_0)$ is regular. 
Then the set 
$$
\CN'_{H}\cap (1,g)\cdot\CN'_H:=\Bigl\{ (\Lambda,\Lambda')\in \CN'_{H}\times \CN'_H\, \big\mid\,
\Lambda'=g\Lambda, \, \Lambda^\flat=\Lambda'^\flat\Bigr\}
$$is finite.
\end{proposition}

\begin{proof} For $\Lambda\in \CN'_H$ we denote $$\CE=\Lambda\cap E^\sharp, \quad \CE^\flat=\Lambda\cap E$$ and denote by $\CE'$ and $\CE'^\flat$ the analogous lattices for $\Lambda'\in \CN'_H$. Note that, by $\Lambda\in \CN'_H$, $\CE$ (resp. $\CE^\flat$) induces a framed flag in $E^\sharp$ (resp. in $E$).

We may and do assume that $g$ is a normal form \eqref{eq:n form U}. 
We may factorize the embedding $\CN'_H\to \CN_P$ into two steps by the two embeddings  \eqref{eq:iota' m} and \eqref{eq:iota' m+1}. Thus we define the fiber product\begin{align}\label{eq:N' H}
\xymatrix{\CN'_{E}\ar[r] \ar[d]& \CN'_P\ar[d]\\   \CN'_{m+1}\ar[r]&\CN'_M.}
\end{align}
We first consider the following set
$$
(\CN'_{E}\times \CN'_E)_g:=\Bigl\{ (\Lambda,\Lambda')\in \CN'_{E}\times \CN'_E\, \big\mid\,
\Lambda'=g\Lambda\Bigr\}.
$$
By definition we have $\CE^\flat\subset \Lambda$ and $\CE'^\flat\subset \Lambda'$. 
We thus obtain
 $$\CE'^\flat+g\CE^\flat\subset \Lambda'.$$
In particular the lattice $\CE'^\flat+g\CE^\flat$ is integral, or equivalently $(\CE'^\flat, g\CE^\flat ) \subset O_F$. 
 Using the normal form  \eqref{eq:n form U} of $g$, it is easy to see that the integrality condition implies that both $\CE^\flat$ and $\CE'^\flat$ are contained in $\varpi^{-N} \CE^{0\flat}$ for some large integer $N$ (depending only on $t_1,\cdots, t_r$ in \eqref{eq:n form U} ). Here $\CE^{0\flat}$ is the fixed lattice in \eqref{eq:CE0}. It follows that both $\CE^\flat$ and $\CE'^\flat$ contain $\varpi ^N\CE^{0\flat}$. Therefore there are only finitely possibilities for the lattices $\CE^\flat,\CE'^\flat$;  moreover we have 
\begin{align}\label{eq:Lam1}
\varpi ^N(\CE^{0\flat}+g\CE^{0\flat})\subset \Lambda'.
\end{align}

Now we return to prove the finiteness of $\CN'_{H}\cap  (1,g)\cdot\CN'_H$. Now we have $e\in \Lambda^\sharp=\Lambda'^\sharp$. With the above result on $(\CN'_{E}\times \CN'_E)_g$, a similar integrality consideration shows that $\Lambda$ contains a lifting of $e$ that differs from $e_{r+1}$ (the fixed lifting of $e$, which is orthogonal to $E+E^\vee$) by an element in $\varpi^{-M}\CE^{0\flat}$ for some large integer $M$ (depending only on $g$). It follows that $\varpi^Me_{r+1}\in \Lambda$ and the same holds for $\Lambda'$. Note that in the normal form, for any $i\geq 1$, the element $g^ie_{r+1}$ is a lifting of $(g^\sharp)^i e$,  and it is orthogonal to $E+E^\vee$. Repeating the argument we may show by induction that 
\begin{align}\label{eq:Lam2}
\varpi^M g^ie_{r+1}\in \Lambda', \quad \forall i\geq 1.
\end{align}

By \eqref{eq:Lam1} and \eqref{eq:Lam2} we conclude that  $\Lambda'$ contains the lattice 
\begin{align}\label{eq:Lam g}
\Lambda(g)=\varpi^M\pair{ e_{r+1},ge_{r+1}, \cdots g^m e_{r+1}}+ \varpi ^N(\CE^{0\flat}+g\CE^{0\flat}).
\end{align} By the regularity of $g$, this last lattice $\Lambda(g)$ has full rank in $V$. By the self-duality of $\Lambda'$, we conclude the desired finiteness of  $\CN'_{H}\cap  (1,g)\cdot\CN'_H$.

\end{proof}

In fact the proof also reduces the ``trivial" case of Liu's conjecture to the Jacquet--Rallis FL.
\begin{proposition}\label{prop:FL triv}
Let $g\in G(V)_\rs$ be a normal form  \eqref{eq:n form U}.
Assume that $\CE^{0\flat}+g\CE^{0\flat}$ is a unimodular lattice in $E+g E$ (equivalently, in terms of  \eqref{eq:n form U}, all $t_i$ are units). Then Conjecture \ref{FLconj} for such $g$ in the $(m,n)$-case reduces to  Conjecture \ref{FLconj} for $g^\sharp$ in the $(m,m+1)$-case.
\end{proposition}
\begin{proof}
Under the assumption, the inclusion \eqref{eq:Lam1} becomes 
$$
\CE^{0\flat}+g\CE^{0\flat}\subset\Lambda'.
$$
It follows that 
$$
\Lambda'=\Lambda'^\sharp\oplus(\CE^{0\flat}+g\CE^{0\flat}),\quad \Lambda=\Lambda^\sharp\oplus(\CE^{0\flat}+g\CE^{0\flat})
$$
and the orbital integral \eqref{eq:orb 1U} becomes
$$
\Orb(g,1_{G(O_{F_0})})=\sum_{(\Lambda^\sharp,\Lambda'^\sharp)\in \CN'_{m+1}\times \CN'_{m+1} \atop
\Lambda'^{\sharp}=g^\sharp\Lambda^\sharp,\, \Lambda^\flat=\Lambda'^\flat} 1.
$$
Similar argument reduces the orbital integral on $S_n$ to that on $S_{m+1}$; we omit the details.

\end{proof}
\begin{remark}
The $(m,m+1)$-case of   Conjecture \ref{FLconj}  is proved by Yun--Gordan (for large $p$) \cite{Y,Go}, and independently by Beuzart-Plessis (for $p>2$) \cite{BP19}.
\end{remark}
\section{Intersection numbers of Bessel cycles}
\label{s:AFL}

 Let $F/F_0$ be an unramified quadratic extension of $p$-adic fields. We assume $F_0=\BQ_p$ for simplicity. Our formulation of the AFL conjecture in the Bessel case will be modeled on the lattice counting interpretation of FL in \S\ref{ss:lat}.
\subsection{The generalized diagonal cycle (or the Bessel cycle)}
\label{ss:Bcycle}
Let $m\in \BZ_{\geq 1}, r\in\BZ_{\geq 0}$ and $n=m+1+2r$. Recall from \S\ref{ss:Nn} the unitary RZ space $\CN_n$, and from \S\ref{ss:fil U} the filtered RZ space $\CN_P$, where $P$ stands for a parabolic subgroup $P_{1^r,m+1,1^r}$ of $G(V)$. 
We have $E_i=\Hom^\circ_{O_F}(\BE,\BX_i)\subset \BV, i=1,\cdots, r$, which are isotropic subspaces of $\BV=\BV_n$. We choose basis $\{e_1,\cdots,e_r\}$ of $E=E_r$ such that
$$E_i=\pair{e_1,\cdots,e_i}_F, \quad  i=1,\cdots, r.
$$ 
Denote by $\BW^\sharp=E^\perp/E$ and fix an isomorphism of $\BW^\sharp$ with $\BV_{m+1}$ defined by \eqref{eq:Vn} with $n$ replaced by $m+1$. Fix a vector $e\in \BW^\sharp$ such that $(e,e)$ is a unit and let $\BW^\sharp=\BW\oplus \pair{e}_F$ be the orthogonal decomposition.  

We recall the morphism \eqref{eq:N MP}
\begin{align}\xymatrix{\CN_{1^r}\ar[r]&\CN_{M_P}:= (\CN_{1})^r\times_{\Spf O_{\bre F}} \CN_{m+1}.}
\end{align}
 Let $\CZ_{m+1}(e)$ be the KR divisor associated to $u\in\BW^\sharp$; then $\CZ_{m+1}(e)\simeq \CN_m$ and hence we have an induced embedding 
 $$
 \xymatrix{\iota\colon \CN_{m}\ar[r]&\CN_{m+1}.}
 $$Similar to the definition of the Bessel subgroup $H$ by \eqref{eq:H def2},
 we define  the fiber product as the Bessel formal scheme $\CN_H$: 
\begin{align}\label{eq:N H}
\xymatrix{\CN_{H}\ar[r] \ar[d]& \CN_P\ar[d]\\ \Spf O_{\bre F}\times_{\Spf O_{\bre F}} \CN_{m}\ar[r]&(\CN_1)^r\times_{\Spf O_{\bre F}}\CN_{m+1}}
\end{align}
where on the bottom row the first factor $\Spf O_{\bre F}\to (\CN_1)^r$ is the neutral component (with the chosen basis), cf. \eqref{eq:N 11}, and the second factor is $\iota$ above.
Let
$$\CN_{m,n}:=\CN_{m}\times_{\Spf O_{\bre F}} \CN_n.
$$
We have a morphism induced by \eqref{eq:N H} and $\CN_P\to \CN_{n}$,
$$
\xymatrix{\CN_{H}\ar[r] &\CN_{m,n}} 
$$
whose image is called the Bessel cycle.
\begin{remark}The Bessel cycle $\CN_H$ may be viewed as  a correspondence (in a loose sense) over the product $\CN_{m}\times_{\Spf O_{\bre F}} \CN_n$. In view of the arithmetic diagonal cycle in the AGGP conjecture \cite{RSZ3}, we may call it the  {\em (local) ``generalized (arithmetic) diagonal cycle"}.
\end{remark}

We need to add a weight factor to the Bessel cycle $\CN_H$ and therefore we need to find a partition of $\CN_H$ into a union of open-and-closed  formal subschemes, similar to \eqref{eq:N 1r}. For this purpose we consider certain auxiliary formal schemes.
Let $P'=P_{1^r,1,m,1^r}$ denote the subgroup $P\times_{G(W^\sharp)}G(W)$ of $P$. (However, note that $P'$ is not a parabolic subgroup of $G(V)$.)  Then we consider the fiber product, denoted by $\CN_{P'}$,
\begin{align}\label{eq:N P1}
\xymatrix{\CN_{P'}\ar[r] \ar[d]& \CN_{P}\ar[d]\\ \CN_{m}\ar[r]^\iota&\CN_{m+1} .}
\end{align}
 In terms of the moduli functor, for $S\in\Nilp$,  $\CN_{P'}(S)$ is the set of $(X,X_\bullet,\iota,\lambda,\rho)$ such that $X_{r+1}/X_r\in\CZ_{m+1}(e)(S)$, i.e., it is a direct product (compatible with the addtional structure)
\begin{align}\label{eq:flat}
X_{r+1}/X_r=\sE_S\times X^\flat.
\end{align}
We have a pull-back of the bottom row 
 \begin{align*}
\xymatrix{ 0\ar[r]& X_r\ar[d]^{=} \ar[r]&X'_{r+1}\ar[r]\ar[d]&\sE_S\ar[d] \ar[r]&0\\ 
0\ar[r]& X_r \ar[r]&X_{r+1}\ar[r]&X_{r+1}/X_r  \ar[r]&0.}
\end{align*}
It follows that there is a morphism 
\begin{align}\label{eq:N P2}
\xymatrix{\CN_{P'} \ar[r]& \CN_{1^{r+1}}}
\end{align}
sending a tuple above to $X'_{r+1}$ with its filtration $0=X_0\subset X_1\subset \cdots\subset X_r\subset X'_{r+1}$.  Let $E^\sharp$ denote the preimage of the line $\pair{u}_F$ under the quotient map $E^\perp\to \BW^\sharp$, and let $e_{r+1}\in E^\sharp$ be a lifting of $u$. Similar to \eqref{eq:N 1r}, there is a natural bijection
\begin{align}\label{eq:N 1r+1}
\CN_{1^{r+1}}\simeq \coprod_{\FL_{E^\sharp}}\Spf O_{\bre F},
\end{align}
where $\FL_{E^\sharp}$ denotes the set of complete flags $\CE_\bullet$ of lattices in $E^\sharp$:
$$
0\subset\CE_1\subset \CE_2\subset\cdots \subset \CE_{r+1}=\CE
$$
such that $\CE_i=\CE\cap E_i$ is of rank $i$. 
Let $\CN_{E^\sharp}$ denote the subfunctor of $\CN_{1^{r+1}}$ defined by 
\begin{align}\label{eq:N E sh}
\CN_{E^\sharp} \simeq \coprod_{\FL^\square_{E^\sharp}}\Spf O_{\bre F},
\end{align}
corresponding to the subset $\FL^\square_{E^\sharp}$  of flags  $\CE_\bullet$ satisfying
$$
\CE_{i}/\CE_{i-1}=\pair{e_i},\quad 1\leq i\leq r+1.
$$There is a natural bijection
\begin{align}\label{eq:fl sq}
\FL^\square_{E^\sharp}\simeq N_{E^\sharp}(F)/N_{E^\sharp}(O_F),
\end{align}
where $N_{E^\sharp}$ denotes the unipotent radical of the Borel subgroup (for the complete flag $E_\bullet$) of $\GL_F(E^\sharp)\simeq \GL_{n,F}$, with the basis $\{e_1,\cdots,e_{r+1}\}$ of $E^\sharp$.

Then the Bessel formal scheme $\CN_{H}$ fits into a Cartesian diagram
\begin{align}\label{eq:N P1}
\xymatrix{\CN_{H}\ar[r] \ar[d]& \CN_{P'}\ar[d]\\ \CN_{E^\sharp}\ar[r]&\CN_{1^{r+1}} .}
\end{align}
The morphism $\CN_{H}\to  \CN_{E^\sharp}$ induces a partition as a disjoint union of open-and-closed formal subschemes
$$
\CN_{H}=\coprod_{\CE_\bullet\in \FL^\square_{E^\sharp}} \CN_{H,\CE_\bullet}.
$$

\begin{proposition}\label{p:fil2KR1}
There is a natural isomorphism:
$$
\xymatrix{\CZ(\CE)^{\dag} \ar[r]^-{\sim}& \CN_{H,\CE_\bullet}.}
$$
\end{proposition}
\begin{proof}
The proof of Prop. \ref{p:fil2KR} applies verbatim, noting that the norm of the special vector $u$ is a unit.

\end{proof}

Let $\CN^\spade_{H, \CE_\bullet}$ be the Zariski closure of $\CN_{H, \CE_\bullet}$ in $\CN_{m,n}$, i.e., the smallest closed formal subscheme of $\CN_{m,n}$ which the morphism $\CN_{H, \CE_\bullet}\to \CN_{n,m}$ factors through. We also denote by $\CN^\spade_{H}$ the {\em disjoint union} of $\CN^\spade_{H, \CE_\bullet}$ over all $\CE_\bullet\in \FL^\square_{E^\sharp}$. Then there is a natural morphism 
$\CN^\spade_{H}\to \CN_{m,n}$.

\begin{remark}
The formal scheme $\CN^\spade_{H, \CE_\bullet}$ seems rather intractable, i.e.,  (to our knowledge) there is no moduli theoretical characterization. 
\end{remark}

Define a character 
$$\xymatrix{\psi:\FL^\square_{E^\sharp}\ar[r] &\BC^\times} 
$$
via the bijection \eqref{eq:fl sq} and the character of $N_{E^\sharp}(F)$ defined by \eqref{eq:ch N}.

The group $(G(\BW)\times G(\BV))(F_0)$ acts on $\CN_{m,n}$. 
We are ready to define the intersection number, for a regular $g\in(G(\BW)\times G(\BV))(F_0)$
\begin{align}\label{eq:def int}
\Int(g):=\sum_{(\CE_\bullet, \CE'_\bullet)\in \FL^\square_{E^\sharp}\times \FL^\square_{E^\sharp} }\psi(\CE_\bullet)  \ov\psi(\CE'_\bullet)\, \bigl(\CN^\spade_{H, \CE_\bullet}, g\cdot  \CN^\spade_{H, \CE'_\bullet}\bigr)_{\CN_{m,n}}.
\end{align}
Here, for two closed formal subscheme $\CZ_1,\CZ_2$ of a regular formal scheme $\CX$ over $\Spf O_{\bre F}$, we define their intersection number as
\begin{align*}
 \bigl(\CZ_1,\CZ_2\bigr)_{\CX}&=\chi ( \CO_{\CZ_1}\otimes^\BL_{\CO_\CX}  \CO_{\CZ_2})\\
 &= \sum_{i,j} (-1)^{i+j}{\rm length}_{O_{\bre F}}{\rm H}^j(\CX, {\rm Tor}_i^{\CO_{\CX}}( \CO_{\CZ_1}, \CO_{\CZ_2} )),
\end{align*}
if the right hand side is finite (e.g., if $\CZ_1\cap\CZ_2$ is a proper scheme over $\Spf O_{\bre F}$). For brevity we write the right hand side of \eqref{eq:def int}
 as
\begin{align}\label{eq:def int 1}
\Int(g)= \bigl(\CN^\spade_{H, \psi}, g\cdot  \CN^\spade_{H,\ov \psi}\bigr)_{\CN_{m,n}}.
\end{align}

\begin{remark}By Prop.~\ref{p:fil2KR1}, we could have simply taken the KR cycle to formulate the intersection problem, avoiding the filtered RZ spaces. However, the filtered RZ spaces seem more natural and amenable for generalization, as the example of Ginzburg--Rallis cycle in \S\ref{s:GR} shows. 
\end{remark}

\subsection{The AFL  conjectures in the Bessel case}\label{ss:AFL}
\begin{conjecture}[Arithmetic Fundamental Lemma conjecture for Bessel cycles]\label{AFLconj}

\smallskip

(a) For a regular $g\in (G(\BW)\times G(\BV))(F_0)$,  the disjoint union 
$$\coprod_{(\CE_\bullet, \CE'_\bullet)\in \FL^\square_{E^\sharp}\times \FL^\square_{E^\sharp}  }\CN^\spade_{H, \CE_\bullet}\cap g\cdot  \CN^\spade_{H, \CE'_\bullet}$$
is a proper scheme over  $\Spf O_{\bre F}$. 

\smallskip

(b)
For a regular $g\in (G(\BW)\times G(\BV))(F_0)$ matching $\gamma\in S_{n}(F_0)_\rs$ where both $g$ and $\gamma$ are in normal forms, we have  
$$ 
\del(\gamma, {\bf 1}_{S_n(O_{F_0})})=-\Int(g)\log q.
$$

\end{conjecture}

We have some indirect evidence towards part (a). 
\begin{proposition}\label{prop:int sch}
Let $g\in G(\BV)(F_0)_\rs$.
Then the intersection $\CN_H\cap (1,g)\cdot\CN_H$ is a noetherian scheme.
\end{proposition}
\begin{proof}
The proof of Prop. \ref{prop:int fin} applies to show that $\CN_H\cap(1, g)\cdot\CN_H$ is contained in $\CN_m\times \CZ(\Lambda(g))$ where  $\Lambda(g)$ is the lattice defined by \eqref{eq:Lam g}. Note that $\CZ(\Lambda(g))$ is a noetherian proper scheme.  Since $\CN_H$ is a locally closed formal subscheme of $\CN_n$ (via the natural morphism  $\CN_H\to\CN_n$), the desired assertion follows.
\end{proof}

In the formulation of the intersection problem in \S\ref{ss:Bcycle}, we have fixed a basis $\{e_1,\cdots, e_r\}$ of $E$. Let $\CE^{0\flat}\subset E$ be the lattice $\pair{e_1,\cdots, e_r}$.  Note that for $g\in G(\BV)_\rs$, the sum $E+gE\subset \BV$ is a non-degenerate Hermitian space with dimension $2r$.
\begin{proposition}
Let $g\in G(\BV)(F_0)$ be a normal form  \eqref{eq:n form U}.
Assume that $\CE^{0\flat}+g\CE^{0\flat}$ is a unimodular lattice in $E+g E$. Then there is an isomorphism of formal schemes
$$
\xymatrix{\CN_H\times_{\CN_{m,n}}(1,g)\cdot\CN_H\ar[r]^-{\sim}& \Delta(\CN_{m})\times_{\CN_{m,m+1}} (1,g^\sharp)\cdot \Delta(\CN_m). }
$$
Here the right hand side is the intersection in the AFL conjecture for $U(m)\times U(m+1)$, i.e. 
$$
\xymatrix{\Delta: \CN_m\ar[r]& \CN_{m,m+1}= \CN_{m}\times_{\Spf O_{\bre F}} \CN_{m+1}}.$$
In particular, if we further assume that $g$ is regular, then $\CN_H\cap(1,g)\cdot\CN_H$ is a proper scheme.

\end{proposition}
\begin{proof}
The proof of Prop.~\ref{prop:FL triv} applies.
\end{proof}

\begin{remark}
One may expect that, if the intersection $\CN_H\cap (1,g)\cdot\CN_H$ in Prop.~ \ref{prop:int sch} is already a proper scheme, then  $\CN_H^\spade\times_{\CN_{m,n}} (1,g)\cdot\CN_H^\spade$ is isomorphic to  $\CN_H\cap (1,g)\cdot\CN_H$.  But we do not know how to prove this.
\end{remark}

There is another special case where we can at least show that the fiber product $\CN_H^\spade\times_{\CN_{m,n}} (1,g)\cdot\CN_H^\spade$ is a proper scheme. This means that the intersection $\CN_{H,\CE_\bullet}^\spade\cap (1,g)\cdot\CN_{H,\CE'_\bullet}^\spade$  is a proper scheme for every pair $ (\CE_\bullet, \CE'_\bullet)\in \FL^\square_{E^\sharp}\times \FL^\square_{E^\sharp}$, and is empty for all but finitely many pairs. In particular, the intersection number in part (b) of the conjecture is well-defined.

\begin{proposition}Suppose $m=1$ (so that $\CN_{m,n}\simeq \CN_n$). 
For $g\in G(\BV)(F_0)_\rs$,
 the formal scheme  $\CN_H^\spade\times_{\CN_{m,n}} (1,g)\cdot\CN_H^\spade$ is a proper scheme.
\end{proposition}
\begin{proof}
By Prop.~\ref{p:fil2KR1}, the Zariski closure $\CN_{H,\CE_\bullet}^\spade$ is a closed formal subscheme of the KR cycle $\CZ(\CE)$ where $\CE=\CE_{r+1}$.
Therefore it suffices to show  that
\begin{enumerate}
\item[(a)]  $\CZ(\CE)\cap g \CZ(\CE')$ is empty except for finitely many pairs $(\CE_\bullet, \CE'_\bullet)\in \FL^\square_{E^\sharp}\times \FL^\square_{E^\sharp} $, and
\item [(b)] $\CZ(\CE)\cap g \CZ(\CE')$ is a proper scheme.
\end{enumerate}

The first assertion follows from the proof of Prop. \ref{prop:int sch} (hence Prop.~\ref{prop:int fin}). For the second one, the proof of   Prop.~\ref{prop:int fin} shows that  $\CE+g\CE'$ is a lattice of full rank in $\BV$.

\end{proof}

\end{document}